\documentclass[a4paper,11pt]{article}
\usepackage{amssymb,amsmath,amsfonts,amsthm}
\usepackage{graphicx,afterpage}
\newcommand{\C}{\ensuremath{\mathcal{C}}}
\newcommand{\sym}{\ensuremath{\mathrm{Sym}}}

\newcommand{\diam}{\ensuremath{\mathrm{Diam} \; }}
\newcommand{\graph}{\ensuremath{\C(G,X)}}
\newcommand{\hatgraph}{\ensuremath{\C(\hat G,\hat X)}}

\newcommand{\ep}{\varepsilon}

\newcommand{\zz}{\mathbb{Z}}

\newcommand{\rr}{\mathbb{R}}

\newcommand{\bu}{\mathbf{u}}
\newcommand{\bv}{\mathbf{v}}
\newcommand{\bw}{\mathbf{w}}
\newtheorem{thm}{Theorem}[section]
\newtheorem{lemma}[thm]{Lemma}
\newtheorem{prop}[thm]{Proposition}

\theoremstyle{definition}
\newtheorem{notation}[thm]{Notation}
\newtheorem{defn}[thm]{Definition}
\begin{document}
    \title{\Large\textbf{Commuting Involution Graphs for $\tilde C_n$}}
\date{}
      \author{Sarah Hart and Amal Sbeiti Clarke\thanks{\noindent Dept of Economics, Mathematics and
      Statistics,
Birkbeck College, Malet Street, London, WC1E 7HX. \hspace*{4mm}
s.hart@bbk.ac.uk,akasso2@mail.bbk.ac.uk }}
  \maketitle
 \vspace*{-10mm}
\begin{abstract}
\noindent In this article we consider commuting graphs of
involution conjugacy classes in the affine Weyl group of type $\tilde
C_n$. We show that where the graph is connected the diameter is at
most $n+2$.
\end{abstract}
 \section{Introduction}
 Consider $X$ a subset of a group $G$. The elements of $X$ constitute the vertices of the  commuting graph \graph, in which $x, y \in X$ are joined by an edge whenever $ xy =  yx$. If $X$ is a set of involutions then we call \graph\ a commuting involution graph. Commuting graphs have been studied by many authors in various contexts. For example Fischer~\cite{cominvsym}  studied commuting involution graphs for the case when $X$ is a conjugacy class of involutions, in his work on  3-transposition groups. Segev and Seitz looked in \cite{segseitz} at commuting graphs for finite simple groups $G$ where $X$ consists of the non-identity elements of $G$. More recently Giudici and Pope~\cite{commuting} gave some results on bounding the diameters of commuting graphs of finite groups. Commuting graphs for elements of order 3 have been considered in \cite{athirah}; there have also been many papers dealing with the case where $X$ consists of all non-identity elements of a given group such as for example \cite{comalg}. \\
  
 In  \cite{cominfinite}, Bates et al looked into the commuting involution graph \graph\ where $X$ is a conjugacy class of involutions in $G$ and $G$ is $\sym(n)$. The remaining finite Coxeter groups were analysed in \cite{Finite}. Commuting involution graphs in affine Coxeter groups of type $\tilde A_n$ have been considered in \cite{Perkins}. In this article, we investigate the commuting involution graphs $C(G,X)$ for type $\tilde{C}_{n}$.\\
 
 For the rest of this article, let $G_n$ be an affine Weyl group of type $\tilde C_n$, for some $n\geq 2$, writing $G$ when $n$ is not specified, and let $X$ be a conjugacy class of involutions of $G$. We write $\diam \graph$ for the diameter of \graph\
 when \graph\ is a connected graph, in other words the maximum distance $d(x,y)$ between any $x, y \in X$ in the graph. Since conjugation by any group element induces a graph automorphism, we can determine the diameter by fixing any $a \in X$,  and then  $\diam \graph = \max\{d(x,a): x \in X\}$. Our main result is the following.
 
 \begin{thm}\label{summary} Let $X$ be a conjugacy class of involutions in the group $G_n$ of type $\tilde C_n$. If the commuting graph \graph\ is connected, then its diameter is at most $n+2$.
 \end{thm}
 
 The proof of Theorem \ref{summary} is broken into several cases depending on the type of the conjugacy class. To describe this further, and in order to state our results about connectedness, we need to describe a parameterisation of the involution conjugacy class. We will explain in Section~2 that involutions in $G$ can be written as `labelled permutations' (Notation \ref{notation}). These are permutations expressed as products of disjoint cycles in which every cycle has a sign and an integer written above it. For example in $G_{10}$ one of the involutions is ${\stackrel{\stackrel{0}{+}}{(1 2)}}{\stackrel{\stackrel{0}{+}}{(3 4)}}\stackrel{\stackrel{3}{-}}{(5)} \stackrel{\stackrel{4}{-}}{(6)}
          \stackrel{\stackrel{3}{-}}{(7)} 
          \stackrel{\stackrel{1}{-}}{(8)} 
                    \stackrel{\stackrel{0}{-}}{(9)} 
                              \stackrel{\stackrel{0}{+}}{(10)}$. A cycle is positive if it has a plus sign, and negative if it has a minus sign. The {\em labelled cycle type} of an involution $\sigma$ will be a quadruple $(m,k_e,k_o,l)$ where $m$ is the number of transpositions, $k_e$ is the number of negative $1$-cycles with an even number above them, $k_o$ is the number of negative 1-cycles with an odd number above them, and $l$ is the number of positive 1-cycles. For the example in $G_{10}$, its labelled cycle type is $(2,2,3,1)$.\\

 We will show (Theorem \ref{conjclass}) that involution conjugacy classes in $G$ are parametrised by labelled cycle type. We may now give the conditions under which the graph \graph\ is connected.
 
 \begin{thm}\label{main}   
 Let $X$ be a conjugacy class of involutions in $G_n$ (where $n \geq 2$) with labelled cycle type $(m,k_e,k_o,l)$. Then \graph\ is disconnected in each of the following cases.
           \begin{trivlist}
    \item[(\textit{i})] $m=0$ and $l=0$;
          \item[(\textit{ii})] $m > 0$, $l=0$ and either $k_e = 1$ or $k_o=1$;
         \item[(\textit{iii})] $m > 0$ and $\max(k_e, k_o, l) = 1$;
         \item[(\textit{iv})] $n=4$ and $m=1$; 
     \item[(\textit{v})] $n=6$, $m=1$ and $k_e =k_o=2$.  
        \end{trivlist}   
 In all other cases, \graph\ is connected.\end{thm}

We note that the bound on diameter given in Theorem \ref{summary} is best possible. For example we have verified that when $n = 8$, $m=1$ and $k_e = k_o = 3$, the commuting involution graph has diameter 10. In Section \ref{sec2} we will establish notation and describe the conjugacy classes of involutions in $G$. Section 3 is dedicated to proving the main theorems. In Section 4 we give examples of selected commuting involution graphs. 
         
\section{Involution conjugacy classes in $G$}\label{sec2}

For the remainder of this paper $a$ is a fixed involution with $X$ its conjugacy class in the affine irreducible Coxeter group $G$. Our first job is to establish what involutions in $G$ look like.\\

Let $W$ be a finite Weyl group with root system $\Phi$ and $\Phi^{\vee}$ the set of coroots. The affine Weyl group $\tilde W$ is the semidirect product of $W$ with translation group $Z$ of the coroot lattice $L(\Phi^{\vee})$  of $W$. See, for example, \cite[Chapter 4]{humphreys} for more detail. For $\alpha, \tau \in W$ and $\bu, \bv \in Z$ we have $$(\alpha,\bv)(\tau,\bu) = (\alpha\tau,\bv^\tau + \bu).$$

For the rest of the paper $W_n$ will denote a Coxeter group of type $C_n$, and set $G_n = \tilde W_n$. We write, respectively, $W$ and $G$ whenever it is not necessary to specify $n$. We may take the roots of $W$ to be $\pm 2e_i$ and $\pm e_i \pm e_j$, for $1 \leq i \leq n$, where $\{e_1, \ldots, e_n\}$ is the standard orthonormal basis for $\rr^n$. The coroots are then $\pm e_i$ and $\pm e_i \pm e_j$. Therefore in this case $$Z  = \{(\lambda_1, \ldots, \lambda_n): \lambda_i \in \zz\}\cong \zz^n.$$ 
We may view the elements of $W$ as signed permutations; they act on $\rr^n$ by permuting the subscripts of basis vectors and changing their signs.  To obtain a signed permutation we write a
 permutation in Sym($n$) (including 1-cycles), add a plus sign or a
 minus sign above each $i$, and say $i$ is positive or negative
 accordingly. We adopt the convention of reading the sign first;
 that is, if $w =(\stackrel{-}{1}\hspace{0.2cm} \stackrel{+}{2}\hspace{0.2cm}\stackrel{-}{3} ) \in W$,
  then $1^w=-2$ , $2^w=3$ and $3^w=-1$.\\

   Expressing $\sigma$ as a product of disjoint cycles, we say that a
   cycle $(i_{1}\cdots i_{r})$ of $\sigma$ is {\em positive} if there is
   an even number of minus signs above its elements, and {\em
   negative} if the number of minus signs is odd. For example, $ (\stackrel{+}{1}\hspace{0.2cm} \stackrel{+}{3}\hspace{0.2cm}\stackrel{-}{2} )$    is a negative cycle, whereas $(\stackrel{-}{5}\;\stackrel{-}{6})$ is
   positive. It is straightforward to check that an
   involution of $W$ only has 1-cycles (positive or
   negative) and positive 2-cycles. By the definition of group multiplication in $G_n$, we see that
     the element $(\sigma, \mathbf{v})$ of $G$ is an involution
     precisely when $(\sigma^2, \mathbf{v}^{\sigma} + \mathbf{v}) =(1,\mathbf{0})$. This allows us to characterise the involutions in $G_n$.

  \begin{lemma} \label{invn}
  A non-identity element $(\sigma, \mathbf{v})$ of $G_n$ is an involution if and only if $\sigma$, when expressed as a product of disjoint signed cycles, has the form
        \[\sigma =(\stackrel{+}{a_1} \overset{+}{b_1})\cdots 
        (\stackrel{+}{a_t} \overset{+}{b_t}) (\stackrel{-}{a_{t+1}} \overset{-}{b_{t+l}}) \cdots (\stackrel{-}{a_m} \overset{-}{b_m}) 
      \stackrel{-}{(c_{2m+1})} \cdots \stackrel{-}{(c_{n-l})} (\stackrel{+}{d_{n-l+1}})\cdots\stackrel{+}{(d_n)}\]
        for some $a_i, b_i, c_i, d_i, t, m$ and $l$; and, writing $\bv = (v_1, \ldots, v_n)$, we have   
       $v_{b_i} = -v_{a_i}$ when $\sigma$ contains  $(\stackrel{+}{a_i} \overset{+}{b_i})$, $v_{b_i} = v_{a_i}$ when $\sigma$ contains $(\stackrel{+}{a_i} \overset{+}{b_i})$ and $v_{d_i} = 0$ for $n-l < i \leq n$.
       \end{lemma}
       \begin{proof} 
   If $a=(\sigma,\mathbf{v})$ is an involution, then 
 $a^2=(\sigma^2, \mathbf{v}^{\sigma} + \mathbf{v}) =(1,\mathbf{0})$. Thus $\sigma$ is an involution of $W$, and so its cycles are all either 1-cycles or positive 2-cycles. Thus $\sigma$ has the form given in the statement of the lemma.
 Write $\mathbf{v}^{\sigma} + \mathbf{v} = (u_1, \ldots, u_n)$. We must have $u_r =0$ for all $r$. For $1 \leq i \leq t$, we have $u_{a_i} = u_{b_i} = v_{b_i} + v_{a_i}$. Thus $v_{b_i} = -v_{a_i}$. For $t < i \leq m$, we have $u_{a_i} = -v_{b_i} + v_{a_i}$ and $u_{b_i} = -v_{a_i} + v_{b_i}$. So $v_{b_i} = v_{a_i}$. For $m < i < n-l$ we have $u_{c_i} = -u_{c_i} + u_{c_i} = 0$, so there is no restriction on $u_{c_i}$. For $n-l < i < n$ we have $u_{d_i} = 2v_{d_i}$, forcing $v_{d_i} = 0$. These are the necessary conditions on $\sigma$ and $\bv$; it is clear that they are also sufficient.
        \end{proof}

 \begin{notation}
\label{notation} 
 From now on we will employ a shorthand for writing involutions $(\sigma,\bv)$ of $G$: if $\bv = (v_1, \ldots, v_n)$, then above each signed number $i$ in the expression of $\sigma$ as a product of disjoint signed cycles, we will write $v_i$. However for transpositions $(\overset{\pm}{a}\; \overset{\pm}{b})$ of $\sigma$, where the number above $a$ determines the number above $b$ as described in Lemma \ref{invn}, we write $(\overset{\stackrel{\lambda}{+}}{a b})$
 for $(\overset{\stackrel{\lambda}{+}}{a} \overset{\overset{-\lambda}{+}}{b})$ and $(\overset{\stackrel{\lambda}{-}}{a b})$
  for $(\overset{\stackrel{\lambda}{-}}{a} \overset{\overset{\lambda}{-}}{b})$. We will call this the {\em labelled cycle form} of $a$. Where it is helpful, we adopt the convention that cycles $(\overset{\overset{0}{+}}{a})$ are omitted, as these fix both $a$ and $v_a$. 
 \end{notation}
 
 \begin{defn}
 Let $a$ be an involution in $G_n$. The {\em labelled cycle type} of $a$ is the tuple $(m,k_e, k_o, l)$, where $m$ is the number of transpositions, $k_e$ is the number of negative 1-cycles with an even number above them, $k_o$ is the number of negative 1-cycles with an odd number above them, and $l$ is the number of positive 1-cycles (fixed points), in the labelled cycle form of $a$. 
 \end{defn}

For example, the labelled cycle type of ${\stackrel{\stackrel{0}{+}}{(1 2)}}\stackrel{\stackrel{1}{-}}{(3)}\stackrel{\stackrel{1}{-}}{(4)}\stackrel{\stackrel{3}{-}}{(5)} \stackrel{\stackrel{4}{-}}{(6)}
          \stackrel{\stackrel{3}{-}}{(7)} 
          \stackrel{\stackrel{1}{-}}{(8)} 
                    \stackrel{\stackrel{0}{+}}{(9)} 
                              \stackrel{\stackrel{0}{+}}{(10)}$ is $(1,1,5,2)$.\\
                              
Having characterised the involutions, we must now determine the conjugacy classes. 
A well-known result, due to Richardson, gives a description of involution conjugacy classes in Coxeter groups.  
 
 \begin{defn} \label{equiv}
   Let $W$ be an arbitrary Coxeter group, with $R$ the set of fundamental reflections. We say that two subsets $I$ and $J$ of $R$ are $W$-equivalent
   if there exists $w \in W$ such that $I^w = J$.
   \end{defn} 
 
 In the next result, we use the notation $w_{I}$ for the longest element of a finite standard parabolic subgroup $W_I$. 
 
  \begin{thm}[Richardson \cite{richardson}]\label{richardson}
    
   Let $W$ be an arbitrary Coxeter group, with $R$ the set of fundamental reflections. Let $g \in W$ be an
   involution. Then there exists $I \subseteq R$ such that $w_I$ is central in $W_I$, and $g$ is conjugate to
   $w_I$. In addition, for $I, J \subseteq R$, $w_I$ is conjugate to $w_J$ if and only if $I$ and $J$ are
   $W$-equivalent.
   \end{thm}
 
 The Coxeter graphs of $G_2 \cong \tilde{C_2}$ and $G_n\cong  \tilde{C_n}$, $n\geq 3$ are as follows. \\
          \unitlength 1.00mm
           \linethickness{0.4pt}
            \begin{picture}(58.00,20)(-10,0)
          \put(10,10.80){\line(1,0){12}}
          \put(10,9.00){\line(1,0){12}}
          \put(10,10.00){\circle*{2.00}}
          \put(22,10.00){\circle*{2.00}}
          \put(34,10.00){\circle*{2.00}}
          \put(10,14){\makebox(0,0)[cc]{$r_{1}$}}
          \put(22,14){\makebox(0,0)[cc]{$r_{2}$}}
          \put(32,14){\makebox(0,0)[cc]{$r_{3}$}}
          \put(-10,10){\makebox(0,0)[lc]{$\tilde C_2$}} \put(22,10.80){\line(1,0){12}}
          \put(22,9.00){\line(1,0){12}}
          \end{picture}
         \unitlength 1.00mm
         \linethickness{0.4pt}
         \begin{picture}(58.00,14)(-30,0)
         \put(10,10.80){\line(1,0){12}}
         \put(10,9.00){\line(1,0){12}}
         \put(22,10.00){\line(1,0){8}}
         \put(50,10.00){\line(1,0){8}}
         \put(58,10.80){\line(1,0){12}}
         \put(58,9.02){\line(1,0){12}}
         \put(58.00,10.00){\circle*{2.00}}
         \put(10,10.00){\circle*{2.00}}
         \put(22,10.00){\circle*{2.00}}
         \put(70,10.00){\circle*{2.00}}
         \put(10,14){\makebox(0,0)[cc]{$r_{1}$}}
         \put(22,14){\makebox(0,0)[cc]{$r_{2}$}}
         \put(58,14){\makebox(0,0)[cc]{$r_{n}$}}
         \put(70,14){\makebox(0,0)[cc]{$r_{n+1}$}}
         \put(-20,10){\makebox(0,0)[lc]  {$\tilde C_n (n\geq 3)$}}
         \put(35.00,10.00){\line(1,0){0.4}}
         \put(45.00,10.00){\line(1,0){0.4}}
         \put(40.00,10.00){\line(1,0){0.4}}
         \end{picture}  
         
  We may set $r_1 = \stackrel{\stackrel{0}{-}}{(1)}$, $r_i = (\overset{\stackrel{0}{+}}{i-1}\;\overset{\stackrel{0}{+}}{i})$ for $2 \leq i \leq n$, and $r_{n+1} = \stackrel{\stackrel{1}{-}}{(n)} $.   \\

It is well known that in the finite Coxeter group $W$ of type $C_n$, elements are conjugate if and only if they have the same signed cycle type. In particular two involutions are conjugate when they have the same number of transpositions, the same number of negative 1-cycles and the same number of positive 1-cycles. \\

In $G$ (which is of type $\tilde C_n$), the element $(\sigma,\mathbf{\bv})$ is conjugate to $(\tau,\mathbf{u})$ via some $(g,\mathbf{w})$ if and only if: \begin{align} (\tau,\mathbf{u}) &= (\sigma, \mathbf{v})^{(g,\mathbf{w})}\nonumber\\
         &= (g^{-1}\sigma g, \mathbf{v}^g + \mathbf{w} - \mathbf{w}^{g^{-1} \sigma g}).\label{eqconj}
         \end{align}

\begin{thm}\label{conjclass}
 Involutions in $G$ are conjugate if and only if they have the same labelled cycle type. In particular, every involution is conjugate to exactly one element $a = a_{m, k_e, k_o, l}$ of the form
$$a={\overset{\overset{0}{+}}{(1\; 2)}} \cdots \overset{\overset{0}{+}}{(2m-1\;2m)}
 \overset{\overset{0}{-}}{(2m+1)}\cdots \overset{\overset{0}{-}}{(2m+k_e)}
 \overset{\overset{1}{-}}{(2m+k_e+1)}\cdots 
 \overset{\overset{1}{-}}{(n-l)}\overset
 {\overset{0}{+}}{(n-l+1)}\cdots 
 \overset{\overset{0}{+}}{(n)}.$$
 \end{thm}
 \begin{proof} Let $x$ be an involution of $G$. By Theorem \ref{richardson},  $x$ is conjugate to
 $w_I$ for some finite standard parabolic
 subgroup $W_I$ of $G$ in which $w_I$ is central. Therefore the connected components of the Coxeter graph for $W_I$ are of types $A_1$ or $B_i$ for some $i$ (including, by a slight abuse of notation, $B_1$, where we have connected components with just the vertex $r_1$ or $r_{n+1}$). Thus $I = \{r_1, r_2, \ldots, r_{i}\} \cup J \cup \{r_{j+1}, r_{j+2}, \ldots, r_{n+1}\}$ for some $i, j$ with  $0 \leq i < j\leq n+1$, where $J$ is a subset of $\{r_{i+2}, \ldots, r_{j-1}\}$ no two elements of which are adjacent vertices in the Coxeter graph. By conjugation in $\langle r_{i+2}, \ldots, r_{j-1}\rangle$ (which after all is isomorphic to the symmetric group $\sym(j-i-1)$), we can assume that for some $i, j$ and $m$ with $0 \leq i \leq i+2m < j < n+1$ we have $$I = \{r_1, r_2, \ldots, r_{i}\} \cup \{r_{i+2}, r_{i+4}, \ldots, r_{i+2m}\} \cup \{r_{j+1}, r_{j+2}, \ldots, r_{n+1}\}.$$ 
 This gives that $x$ is conjugate to $w_I$, where
 $$w_I =   \overset{\overset{0}{-}}{(1)}\cdots \overset{\overset{0}{-}}{(i)}{\overset{\overset{0}{+}}{(i + 1\;\; i+2)}} \cdots \overset{\overset{0}{+}}{(i + 2m-1\;\;\;i + 2m)}
  \overset{\overset{0}{+}}{(i+2m+1)}\cdots 
  \overset{\overset{0}{+}}{(j-1)}\overset
  {\overset{1}{-}}{(j)}\cdots 
  \overset{\overset{1}{-}}{(n)}.$$
  Let $c = (h, \mathbf{0})$, where
\begingroup\makeatletter\def\f@size{10}\check@mathfonts
$$ h = (\overset{+}{1}\;\; \overset{+}{i+1})(\overset{+}{2}\;\; \overset{+}{i+2})\cdots (\overset{+}{2m}\;\; \overset{+}{i+2m})(\overset{+}{i+2m+1}\;\;\; \overset{+}{j})\;(\overset{+}{i+2m+2} \;\;\;\overset{+}{j+1})\cdots (\overset{+}{i+2m+1+n-j}\;\;\; \overset{+}{n}).$$ \endgroup  
From Equation \eqref{eqconj} we see that 
 $$w_I^c ={\overset{\overset{0}{+}}{(1\; 2)}} \cdots \overset{\overset{0}{+}}{(2m-1\;2m)}
  \overset{\overset{0}{-}}{(2m+1)}\cdots \overset{\overset{0}{-}}{(2m+k_e)}
  \overset{\overset{1}{-}}{(2m+k_e+1)}\cdots 
  \overset{\overset{1}{-}}{(n-l)}\overset
  {\overset{0}{+}}{(n-l+1)}\cdots 
  \overset{\overset{0}{+}}{(n)}.$$
  Therefore, by setting $a = w_I^c$ we see that each involution in $G$ is conjugate to at least one element of the required form.\\ 
  
  Now consider an involution $x$ in $G$ with labelled cycle type $(m,k_e,k_o,l)$, and suppose $y = r_ixr_i^{-1}$ for some simple reflection $r_i$. Write $x = (\sigma, \bv)$ and $y = (\tau, \bu)$. By Equation \eqref{eqconj},  $\tau$ is conjugate to $\sigma$ in the underlying Weyl group $W$. Hence $\tau$ and $\sigma$ have the same number of transpositions, negative 1-cycles and positive 1-cycles as each other. In other words, the labelled cycle type of $y$ is $(m,k'_e, k'_o,l)$ for some $k'_e, k'_o$ satisfying $k'_e + k'_o = k_e + k_o$. Now, from Equation \eqref{eqconj} again, $(\overset{\overset{\lambda}{-}}{i})$ is a labelled 1-cycle of $x$ if and only if $(\overset{\overset{\mu}{-}}{i^g})$ is a labelled 1-cycle of $y$, where $\lambda = v_i$ and $\mu = u_{i^g}$.  In particular this means $(\overset{-}{i^g})$ is a signed cycle of $\tau$, so that $(w^{\tau})_{i^g} = -w_{i^g}$.
  Now $\bu = \bv^g + \bw - \bw^{\tau}$. So $
\mu = u_{i^g} = (v^g)_{i^g} + w_{i^g} - (w^\tau)_{i^g} = v_i + 2w_{i^g} \equiv \lambda \mod{2}$.
  Therefore $k_o = k'_o$. Hence $k_e = k'_e$ and so $x$ and $y$ have the same labelled cycle type. In particular $x$ is conjugate to at most one element $a$ of the form stated in the theorem. Conversely, any two involutions of the same labelled cycle type $(m,k_e,k_o,l)$ are both conjugate to $a_{m,k_e,k_o,l}$, and hence to each other. Thus conjugacy is parameterised by labelled cycle type, and the set of elements $\{a_{m,k_e,k_o,l} : 2m + k_e + k_o + l = n\}$ contains exactly one representative of each conjugacy class of involutions in $G_n$. \end{proof}

%
  We next prove three preliminary lemmas which will be used repeatedly in the proofs in Section~\ref{sec3}.

\begin{lemma}\label{1cycle}
Let $\alpha \in \{1, \ldots, n\}$ and $\lambda, \mu \in \zz$. Then $\stackrel{\stackrel{\lambda}{-}}{(\alpha)}$ commutes with $\stackrel{\stackrel{0}{+}}{(\alpha)}$ for all $\lambda$, whereas
$\stackrel{\stackrel{\lambda}{-}}{(\alpha)}$ commutes with  $\stackrel{\stackrel{\mu}{-}}{(\alpha)}$ if and only if $\lambda = \mu$. 
\end{lemma}

  \begin{proof}
  	For convenience, we assume without loss of generality that $\alpha = 1$. Now $\stackrel{\stackrel{\lambda}{-}}{(1)}\stackrel{\stackrel{0}{+}}{(1)} = \stackrel{\stackrel{\lambda}{-}}{(1)} = \stackrel{\stackrel{0}{+}}{(1)}\stackrel{\stackrel{\lambda}{-}}{(1)}$, so $\stackrel{\stackrel{\lambda}{-}}{(1)}$ commutes with $\stackrel{\stackrel{0}{+}}{(1)}$. In the case of two negative 1-cycles we have $\stackrel{\stackrel{\lambda}{-}}{(1)}\stackrel{\stackrel{\mu}{-}}{(1)} = \stackrel{\stackrel{\mu-\lambda}{+}}{(1)}$ and $\stackrel{\stackrel{\mu}{-}}{(1)}\stackrel{\stackrel{\lambda}{-}}{(1)} = \stackrel{\stackrel{\lambda-\mu}{+}}{(1)}$, so $\stackrel{\stackrel{\lambda}{-}}{(1)}$ commutes with $\stackrel{\stackrel{\mu}{-}}{(1)}$ if and only if $\lambda  = \mu$.
  \end{proof}
   
    \begin{lemma} \label{2cycle} Let $\alpha, \beta$ be distinct elements of  $\{1, \ldots, n\}$  and let $\lambda$, $\mu$ and $\nu$ be integers. 
    \begin{trivlist}
\item[(\textit{i})]  $\overset{\overset{\lambda}{+}}{(\alpha\hspace{0.1cm} \beta)}$ and
      $\overset{\overset{\mu}{+}}{(\alpha\hspace{0.1cm} \beta)}$ commute if and only if $\lambda = \mu$, and $\overset{\overset{\lambda}{-}}{(\alpha\hspace{0.1cm} \beta)}$ and
      $\overset{\overset{\mu}{-}}{(\alpha\hspace{0.1cm} \beta)}$ commute if and only if $\lambda = \mu$. But $\overset{\overset{\lambda}{+}}{(\alpha\hspace{0.1cm} \beta)}$ and
         $\overset{\overset{\mu}{-}}{(\alpha\hspace{0.1cm} \beta)}$ commute for all $\lambda$ and $\mu$.
      \item[(\textit{ii})]
                $\overset{\overset{\lambda}{\pm}}{(\alpha\hspace{0.1cm} \beta)}$ and  $\overset{\overset{0}{+}}{(\alpha)}
                 \overset{\overset{0}{+}}{(\beta)}$ commute for all $\lambda$, but there is no value of $\mu$ or $\lambda$ for which  $\overset{\overset{\lambda}{\pm}}{(\alpha\hspace{0.1cm} \beta)}$ and  $\overset{\overset{0}{+}}{(\alpha)}
                 \overset{\overset{\mu}{-}}{(\beta)}$ or $\overset{\overset{\mu}-}{(\alpha)}
                \overset{\overset{0}{+}}{(\beta)}$ commute.
                  \item[(\textit{iii})]
           $\overset{\overset{\lambda}
            {+}}{(\alpha\hspace{0.1cm} \beta)}$ and $\overset{\overset{\mu}{-}}{(\alpha)}
          \overset{\overset{\nu}{-}}{(\beta)}$ commute if and only if $\mu-\nu=2\lambda$, whereas 
           $\overset{\overset{\lambda}{-}}{(\alpha\hspace{0.1cm} \beta)}$ and  $\overset{\overset{\mu}{-}}{(\alpha)}
          \overset{\overset{\nu}{-}}{(\beta)}$ commute if and only if $\mu+\nu=2\lambda$.   
    \end{trivlist}  
     
     \end{lemma}

\begin{proof} We lose no generality by assuming, for ease of notation, that $\alpha = 1$, $\beta = 2$ and $n=2$. Recall that $\overset{\overset{\lambda}{+}}{(1\hspace{0.1cm} 2)}$ and $\overset{\overset{\lambda}{-}}{(1\hspace{0.1cm} 2)}$ are  shorthand for $(\overset{+}{(1\hspace{0.1cm} 2)}, (\lambda, -\lambda))$ and $(\overset{-}{(1\hspace{0.1cm} 2)}, (\lambda, \lambda))$ respectively. Involutions commute precisely when their product is an involution (or the identity), which we can check using Lemma \ref{invn}.   \begin{trivlist}
		\item[(\textit{i})]
We calculate
$$\overset{\overset{\lambda}{+}}{(1\hspace{0.1cm} 2)}\overset{\overset{\mu}{+}}{(1\hspace{0.1cm} 2)} =  (\overset{+}{(1\hspace{0.1cm} 2)}, (\lambda, -\lambda))(\overset{+}{(1\hspace{0.1cm} 2)}, (\mu, -\mu)) = (1, (\lambda, -\lambda)^{\overset{+}{(1\hspace{0.1cm} 2)}} + (\mu, -\mu)) = (1, (\mu-\lambda, \lambda-\mu)).$$ 
This means $\overset{\overset{\lambda}{+}}{(1\hspace{0.1cm} 2)}$ and $\overset{\overset{\mu}{+}}{(1\hspace{0.1cm} 2)}$ commute if and only if $\lambda = \mu$. We also have 
$$\overset{\overset{\lambda}{-}}{(1\hspace{0.1cm} 2)}\overset{\overset{\mu}{-}}{(1\hspace{0.1cm} 2)} =  (\overset{-}{(1\hspace{0.1cm} 2)}, (\lambda, \lambda))(\overset{-}{(1\hspace{0.1cm} 2)}, (\mu, \mu)) = (1, (\lambda, \lambda)^{\overset{-}{(1\hspace{0.1cm} 2)}} + (\mu, \mu)) = (1, (\mu-\lambda, \mu-\lambda)),$$ 
         so again $\overset{\overset{\lambda}{-}}{(1\hspace{0.1cm} 2)}$ and $\overset{\overset{\mu}{-}}{(1\hspace{0.1cm} 2)}$ commute if and only if $\lambda = \mu$. Finally a similar calculation shows that
         $\overset{\overset{\lambda}{+}}{(1\hspace{0.1cm} 2)}\overset{\overset{\mu}{-}}{(1\hspace{0.1cm} 2)} = \overset{\overset{\mu +\lambda}{-}}{(1)} \;\overset{\overset{\mu-\lambda}{-}}{(2)}$, which is an involution for all values of $\lambda$ and $\mu$, so $\overset{\overset{\lambda}{+}}{(1\hspace{0.1cm} 2)}$ and $\overset{\overset{\mu}{-}}{(1\hspace{0.1cm} 2)}$ always commute.
\item[(\textit{ii})]      Certainly     $\overset{\overset{\lambda}{\pm}}{(1\hspace{0.1cm} 2)}$ and  $\overset{\overset{0}{+}}{(1)}
\overset{\overset{0}{+}}{(2)}$ commute for all $\lambda$; after all,  $\overset{\overset{0}{+}}{(1)}
\overset{\overset{0}{+}}{(2)}$ is just the identity element when $n=2$. Moreover, 
        already in the underlying group $W$ of type $B_2$ we observe that $\overset{\pm}{(1\hspace{0.1cm} 2)}$ does not commute with  $\overset{+}{(1)}
         \overset{-}{(2)}$  or $\overset{-}{(1)}
         \overset{+}{(2)}$, so there is no value of $\mu$ or $\lambda$ for which  $\overset{\overset{\lambda}{\pm}}{(1\hspace{0.1cm} 2)}$ and  $\overset{\overset{0}{+}}{(1)}
         \overset{\overset{\mu}{-}}{(2)}$ or $\overset{\overset{\mu}-}{(1)}
         \overset{\overset{0}{+}}{(2)}$ commute.
\item[(\textit{iii})] We have         
           $\overset{\overset{\lambda}
           	{+}}{(1\hspace{0.1cm} 2)}\overset{\overset{\mu}{-}}{(1)}
           \overset{\overset{\nu}{-}}{(2)} = (\overset{\overset{\mu - \lambda}{-}}{1}
           \; \overset{\overset{\nu + \lambda}{-}}{2})$; this is an involution if and only if $\nu + \lambda = \mu - \lambda$. Thus $\overset{\overset{\lambda}
           	{+}}{(1\hspace{0.1cm} 2)}$ and $\overset{\overset{\mu}{-}}{(1)}
           \overset{\overset{\nu}{-}}{(2)}$ commute if and only if $\mu-\nu=2\lambda$. Similarly  
                      $\overset{\overset{\lambda}
                      	{-}}{(1\hspace{0.1cm} 2)}\overset{\overset{\mu}{-}}{(1)}
                      \overset{\overset{\nu}{-}}{(2)} = (\overset{\overset{\mu - \lambda}{+}}{1}
                      \; \overset{\overset{\nu - \lambda}{+}}{2})$; this is an involution if and only if $\nu - \lambda = -(\mu - \lambda)$. Therefore $\overset{\overset{\lambda}
                      	{-}}{(1\hspace{0.1cm} 2)}$ and $\overset{\overset{\mu}{-}}{(1)}
                      \overset{\overset{\nu}{-}}{(2)}$ commute if and only if $\mu+\nu=2\lambda$. \qedhere\end{trivlist}
         \end{proof}

          \begin{lemma}\label{doubletrans} Let 
          $g_1=\overset{\overset{\lambda_1}{+}}{(\alpha\beta)}\overset{\overset{\lambda_2}{+}}{(\gamma\delta)}$, $g_2=\overset{\overset{\lambda_1}{+}}{(\alpha\beta)}\overset{\overset{\lambda_2}
          	{-}}{(\gamma\delta)}$, $g_3 = \overset{\overset{\lambda_1}{-}}{(\alpha\beta)}\overset{\overset{\lambda_2}
          	{-}}{(\gamma\delta)}$, $h_1=\overset{\overset{\mu_1}{+}}{(\alpha \gamma)}\overset{\overset{\mu_2}{+}}{(\beta\delta)}$, $h_2=\overset{\overset{\mu_1}{+}}{(\alpha \gamma)}\overset{\overset{\mu_2}{-}}{(\beta\delta )}$
          and 
          $h_3=\overset{\overset{\mu_1}{-}}{(\alpha \gamma)}\overset{\overset{\mu_2}{-}}{(\beta\delta )}$,   for distinct $\alpha,
          \beta, \gamma, \delta$ in $\{1, \ldots, n\}$ and integers
          $\lambda_i, \mu_i$. Then
          
          \begin{trivlist}
          \item[(i)] $g_1$ commutes with $h_1$ if and only if $\mu_1 -\lambda_1 =
          \mu_2-\lambda_2$;  \item[(ii)] $g_1$ does not commute with $h_2$; \item[(iii)] $g_1$ commutes with $h_3$ if and only if $\mu_1-\lambda_1 = \mu_2 +\lambda_2$; \item[(iv)]  $g_2$ commutes with $h_2$  if and only if $\mu_1 -\lambda_1 = \mu_2-\lambda_2$; 
          \item[(v)] $g_2$ does not commute with $h_3$;
          \item[(vi)] $g_3$ commutes with $h_3$ if and only if $\mu_1 - \lambda_1 = \lambda_2 - \mu_2$.
          \end{trivlist}
          \end{lemma}
          \begin{proof} We lose no generality by assuming, for
          ease of notation, that $n=4$ and  $g_1=\overset{\overset{\lambda_1}{+}}
          {(1\hspace{0.1cm} 2)}\overset{\overset{\lambda_2}{+}}
          {(3\hspace{0.1cm} 4)}$ 
           
          \begin{trivlist}
          	\item[(i)] $g_1 =(\overset{+}{(1\hspace{0.1cm}2)}\overset{+}{(3\hspace{0.1cm}4)},(\lambda_1,
          	-\lambda_1, \lambda_2, -\lambda_2))$ and $h_1 = (\overset{+}{(1\hspace{0.1cm}3)}\overset{+}{(2\hspace{0.1cm}4)}, (\mu_1,
          	\mu_2, -\mu_1, -\mu_2))$. Hence
          	\begin{align*}
          		g_1h_1&= (\overset{+}{(1\hspace{0.1cm}4)}\overset{+}{(2\hspace{0.1cm}3)}, (\lambda_1, -\lambda_1, \lambda_2,
          		-\lambda_2)^{(\overset{+}{(1\hspace{0.1cm}3)}\overset{+}{(2\hspace{0.1cm}4)}} + (\mu_1, \mu_2, -\mu_1, -\mu_2))\\
          		&= (\overset{+}{(1\hspace{0.1cm}4)}\overset{+}{(2\hspace{0.1cm}3)}, (\mu_1 +\lambda_2, \mu_2-\lambda_2, -\mu_1 +
          		\lambda_1,
          		-\mu_2-\lambda_1)).\end{align*} %
          	Now $g_1$ and $h_1$ commute if and only if $g_1h_1$ is an involution. This occurs if and only if $\mu_1 + \lambda_2 =
          	-(\mu_2-\lambda_1)$ and $ \mu_2-\lambda_2 = -(-\mu_1 + \lambda_1)$. Rearranging gives $\mu_1 -\lambda_1 = \mu_2-\lambda_2$, as required for part (i). 
\item[(ii)] Since $\overset{+}{(1\hspace{0.1cm}2)}\overset{+}{(3\hspace{0.1cm}4)}$ and $\overset{+}{(1\hspace{0.1cm}3)}\overset{-}{(2\hspace{0.1cm}4)}$ do not commute, it is impossible for $g_1$ to commute with $h_2$. 
\item[(iii)] We have $g_1 =(\overset{+}{(1\hspace{0.1cm}2)}\overset{+}{(3\hspace{0.1cm}4)},(\lambda_1,
            -\lambda_1, \lambda_2, -\lambda_2))$ and $h_3 = (\overset{-}{(1\hspace{0.1cm}3)}\overset{-}{(2\hspace{0.1cm}4)}, (\mu_1,
            \mu_2, \mu_1, \mu_2))$, and \begin{align*}
            g_1g_3&= (\overset{-}{(1\hspace{0.1cm}4)}\overset{-}{(2\hspace{0.1cm}3)}, (\lambda_1, -\lambda_1, \lambda_2,
            -\lambda_2)^{\overset{-}{(1\hspace{0.1cm}3)}\overset{-}{(2\hspace
            {0.1cm}4)}} + (\mu_1, \mu_2, \mu_1, \mu_2))\\
            &= (\overset{-}{(1\hspace{0.1cm}4)}\overset{-}{(2\hspace{0.1cm}3)}, (\mu_1 -\lambda_2, \mu_2+\lambda_2, \mu_1 -
            \lambda_1,\mu_2+\lambda_1)).\end{align*} %
            Now $g_1$ and $h_3$ commute if and only if $g_1h_3$ is an
            involution. From the above calculation this occurs if and only if $\mu_1 - \lambda_2 =
            \mu_2+\lambda_1$ and $ \mu_2+\lambda_2 = \mu_1-\lambda_1$. That is, if  and only if $ \mu_2+\lambda_2 =\mu_1-\lambda_1$.
\item[(iv)] We calculate $g_2h_2 = (\overset{-}{(1\hspace{0.1cm}4)}\overset{+}{(2\hspace{0.1cm}3)}, (\mu_1 +\lambda_2, \mu_2-\lambda_2, -\mu_1 +
            \lambda_1,\mu_2+\lambda_1))$. Thus $g_2$ commutes with $h_2$ if and only if $\mu_1 - \lambda_1 = \mu_2 - \lambda_2$.
\item[(v)] Since $\overset{+}{(1\hspace{0.1cm}2)}\overset{-}{(3\hspace{0.1cm}4)}$ and $\overset{-}{(1\hspace{0.1cm}3)}\overset{-}{(2\hspace{0.1cm}4)}$ do not commute, $g_2$ cannot commute with $h_3$.
\item[(vi)] This part follows from the fact that $g_3h_3 = (\overset{+}{(1\hspace{0.1cm}4)}\overset{+}{(2\hspace{0.1cm}3)}, (\mu_1 -\lambda_2, \mu_2-\lambda_2, \mu_1 -
            \lambda_1,\mu_2-\lambda_1)$.  \qedhere
       \end{trivlist} \end{proof}

     We end this section by stating the relevant results from \cite{Finite} about commuting involution graphs in Weyl groups of type $B_n$. 
                   \begin{thm}[Theorem 1.1 of \cite{Finite}] \label{finite}
                    Suppose that $W$ is of type $B_n$, and let $$\sigma=\overset{+}{(12)}\cdots \overset{+}
                    {(2m-1 \;\; 2m)}\overset{+}{(2m+1)}\cdots \overset{+}{(2m+l)}\overset{-}
                    {(2m+l+1)}
                    \cdots\overset{-}{(2m+l+t)}.$$ Set $X=a^{G}$ and $k:=\max\{l,t\}$. Then the following hold.
                    \begin{trivlist}
                    \item[(\textit{i})] If $m=0$, then $\C(G,X)$ is a complete graph.
                    \item[(\textit{ii})] If $k=0$, then $\diam \C(G,X) \leq 2$.
                    \item[(\textit{iii})] If $k=1$ and $m>0$, then $\C(G,X)$ is disconnected.
                    \item[(\textit{iv})] If $k \geq 2$ and $n>5$, then $\diam \C(G,X) \leq 4$.
                    \item[(\textit{v})] If $n=5$, $m=1$ and $k=2$ then $\diam \C(G,X)=5$. If $n=5$, $m=1$ and $k=3$ then $\diam \C(G,X)=2$. Finally if $n=4$, $m=1$ and $k=2$ then $\C(G,X)$ is disconnected.
                    \end{trivlist}
                  \end{thm}

   \section{Proof of Main Results} \label{sec3}

We begin by looking at connectedness. For an element $g = (\sigma, \bv)$ in a conjugacy class $X$ of $G$, define $\hat g = \sigma$. Then let $\hat X$ be the conjugacy class of $\hat g$ in $W$. Clearly if $g, h \in X$,
       then $\hat g, \hat h \in \hat X$. 
      \begin{lemma} \label{easy} Suppose $g, h \in X$. If $d(\hat g,
      \hat h) = k$, then $d(g,h) \geq k$. If $\hatgraph$ is disconnected,
      then $\graph$ is disconnected.
      \end{lemma}
      \begin{proof} The result follows immediately from the observation that is $g$ commutes with $h$ in $G$, then $\hat g$ commutes with $\hat h$ in $W$. 
      \end{proof}

We can now prove Theorem \ref{main} (which gives necessary and sufficient conditions for \graph\ to be disconnected) in one direction. The proof in the other direction will arise from bounding the diameters of graphs not shown in Theorem \ref{mainhalf} to be disconnected. 

\begin{thm}\label{mainhalf}   
 Let $X$ be a conjugacy class of involutions in $G_n$ (where $n \geq 2$) with labelled cycle type $(m,k_e,k_o,l)$. Then \graph\ is disconnected in each of the following cases.
           \begin{trivlist}
    \item[(\textit{i})] $m=0$ and $l=0$;
            \item[(\textit{ii})] $m > 0$, $l=0$ and either $k_e = 1$ or $k_o=1$;
           \item[(\textit{iii})] $m > 0$ and $\max(k_e, k_o, l) = 1$;
           \item[(\textit{iv})] $n=4$ and $m=1$; 
       \item[(\textit{v})] $n=6$, $m=1$ and $k_e =k_o=2$.  
        \end{trivlist}   
\end{thm}

\begin{proof} Let $X$ be a conjugacy class of involutions in $G_n$, with labelled cycle type $(m, k_e, k_o, l)$, and $a = a_{m,k_e,k_o,l} \in X$ as defined in Theorem \ref{conjclass}. We deal with each case in turn.
 \begin{trivlist}
    \item[(\textit{i})] Suppose $l=0$ and $m=0$. Then any $x$ in $X$ is of the form  $\stackrel{\stackrel{\lambda_1}{-}}{(1)}\stackrel{\stackrel{\lambda_2}{-}}{(2)} \cdots \stackrel{\stackrel{\lambda_n}{-}}{(n)}$ for some $\lambda_i$ (where $k_e$ of the $\lambda_i$ are even and $k_o$ are odd). By Lemma  \ref{1cycle}, $x$ does not commute with any other element of $X$. So in fact \graph\ is completely disconnected in this case. 
          \item[(\textit{ii})] Suppose $l=0$ and $k_o = 1$ (the case $k_e = 1$ is similar). Then $$a = \stackrel{0}{\stackrel{+}{(1\; 2)}}\cdots \stackrel{0}{\stackrel{+}{(2m-1\;\; 2m)}}\stackrel{0}{\stackrel{-}{(2m+1)}}\cdots \stackrel{0}{\stackrel{-}{(n-1)}}\stackrel{1}{\stackrel{-}{(n)}}.$$
          Suppose $b \in X$ such that $a$ commutes with $b$. Consider the cycle of $b$ that contains $n$. This must be a negative 1-cycle or a 2-cycle, because $b$ has the same labelled cycle type as $a$. If it is a 2-cycle, then $b$ cannot commute with $a$, by Lemma \ref{2cycle}(iii). Therefore it is a negative 1-cycle  $\stackrel{\lambda}{\stackrel{-}{(n)}}$ where $\lambda$ is odd, and then by Lemma \ref{1cycle}, $b$ contains  $\stackrel{1}{\stackrel{-}{(n)}}$. The same argument shows that any element $c$ of $X$ which commutes with $b$ must also contain  $\stackrel{1}{\stackrel{-}{(n)}}$, and inductively all elements in the connected component of \graph\ containing $a$ must contain  $\stackrel{1}{\stackrel{-}{(n)}}$. Therefore \graph\ is disconnected.
          
         \item[(\textit{iii})] Suppose $m>0$ and $\max(k_e, k_o, l) = 1$. If $l=0$ then \graph\ is disconnected by (ii). So we can assume $l=1$. If either of $k_e$ or $k_o$ is zero, then Theorem~\ref{finite}(iii) and Lemma~\ref{easy} imply that \graph\ is disconnected. It remains to consider the case $l=k_e = k_o = 1$. Here, $2m = n-3$ and  $$a = \stackrel{0}{\stackrel{+}{(1\; 2)}}\cdots \stackrel{0}{\stackrel{+}{(2m-1\;\; 2m)}}\stackrel{0}{\stackrel{-}{(n-2)}}\stackrel{1}{\stackrel{-}{(n-1)}}\stackrel{0}{\stackrel{+}{(n)}}.$$
         Suppose $b \in X$ such that $a$ commutes with $b$, and suppose $\stackrel{\mu}{\stackrel{\pm}{(t)}}$ is a 1-cycle of $b$. If $t \in \{1, \ldots, 2m\}$ then as $a$ contains a transposition $\stackrel{0}{\stackrel{\pm}{(t t')}}$ for some $t'$, Lemma \ref{2cycle} (ii) and (iii) show that the only way $a$ and $b$ could commute is if $b$ contained $\stackrel{0}{\stackrel{+}{(t)}}\stackrel{0}{\stackrel{+}{(t')}}$ or $\stackrel{\mu}{\stackrel{-}{(t)}}\stackrel{\nu}{\stackrel{-}{(t)}}$ where $\mu \equiv \nu \; (\text{mod } 2)$, contradicting our assumptions about the labelled cycle type of $b$. Therefore the elements appearing in 1-cycles of $b$ are $n-2$, $n-1$ and $n$. Inductively this holds for all elements in the connected component of \graph\ containing $a$. Therefore \graph\ is disconnected.
         \item[(\textit{iv})] Suppose $n=4$ and $m=1$. If $\max(k_e, k_o, l) = 1$ then \graph\ is disconnected by (iii). If this doesn't happen, then one of $k_e$, $k_o$ or $l$ is 2. By Theorem~\ref{finite}(v) and Lemma~\ref{easy}, \graph\ is again disconnected. 
     \item[(\textit{v})] Suppose $n=6$, $m=1$ and $k_e =k_o=2$. For any $x$ in $X$, we have $x = \stackrel{\lambda}{\stackrel{+}{(\alpha \; \alpha')}}\stackrel{\mu}{\stackrel{-}{(\beta)}}\stackrel{\mu'}{\stackrel{-}{(\beta')}}\stackrel{\nu}{\stackrel{-}{(\gamma)}}\stackrel{\nu'}{\stackrel{-}{(\gamma')}}$, where $\{\alpha, \alpha', \beta, \beta', \gamma, \gamma'\} = \{1, 2, 3, 4, 5, 6\}$, $\mu$ and $\mu'$ are even, and $\nu, \nu'$ are odd. Associate a set $T(x) = \left\{\{\alpha, \alpha'\}, \{\beta, \beta'\}, \{\gamma, \gamma'\}\right\}$ to $x$. Given that, by Lemma \ref{2cycle}, a transposition can only commute with a pair of negative 1-cycles if either both cycles are odd or both are even, and also with reference to Lemma \ref{1cycle}, we see that if $y$ in $X$ commutes with $x$, then $T(x) = T(y)$. Therefore, for example,  $\overset{\overset{0}{+}} {(1\;2)}
          \overset{\overset{0}{-}} {(3)}\overset{\overset{0}{-}} {(4)}
        \overset{\overset{1}{-}} {(5)}
          \overset{\overset{1}{-}} {(6)}$
           and $\overset{\overset{0}{+}} {(4\;5)}
           \overset{\overset{0}{-}}{(1)}\overset{\overset{1}{-}} {(2)}
           \overset{\overset{0}{-}} {(3)}    
            \overset{\overset{1}{-}} {(6)}$ are not connected in \graph.\qedhere 
        \end{trivlist}   
\end{proof}

Our first result bounding diameters is when $m=0$. 

\begin{thm}\label{disc}
	If $m=0$ and $l\geq 1$, then $$ \diam \graph  =
	\left\{\begin{array}{ll}\textstyle 2& \text{if  $2l\geq n$} \\ \lceil\frac{n}{l}\rceil & \it{if } 0 \in \{k_e, k_o\} \\
	n + 1 & \text{if $l = 1$ and $0 \notin \{k_e, k_o\}$}.\end{array} \right.$$
	In all other cases $\lceil\frac{n}{l}\rceil 
	\leq \diam \graph \leq 
	\lceil\frac{n}{l}\rceil + 2$.
\end{thm}

\begin{proof}
 In this case we have $$a=
    \overset{\overset{0}{-}}{(1)}\cdots \overset{\overset{0}{-}}{(k_e)}\overset{\overset{1}{-}}{(k_e+1)}\cdots  \overset{\overset{1}{-}}{(k_e+k_o)}\overset{\overset{0}{+}}{(k_e+k_o+1)}\cdots \overset{\overset{0}{+}}{(n)}.$$ 
 Let $g\in X$. Then for appropriate $\varepsilon_1, \ldots, \varepsilon_{k_e+k_o}, \rho_1, \ldots, \rho_{l}$ and $\lambda_1, \ldots, \lambda_{k_e+k_o}$  we have that
  $g=\prod_{i=1}^{k_e+k_o}\overset{\overset{\lambda_i}{-}}{(\varepsilon_i)}\prod_{i=1}^{l}\overset
  {\overset{0}{+}}{(\rho_i)}$. Since $g \in X$, exactly $k_e$ of the $\lambda_i$ must be even.  
If $2l \geq n$, that is, $l > k_e + k_o$, then $g$ commutes with an element $h$ in $X$ which contains the cycles $\overset{\overset{0}{+}}{(1)} \cdots \overset{\overset{0}{+}}{(k_e + k_e)}$. Now $h$ certainly commutes with $a$, and so $\diam \graph = 2$.\\

Now we suppose that $2l < n$, and that one of $k_e$ and $k_o$ is zero. Without loss of generality, we may choose $k_o = 0$. Then, writing $k=k_e$, we have $a=
    \overset{\overset{0}{-}}{(1)}\cdots \overset{\overset{0}{-}}{(k)}\overset{\overset{0}{+}}{(k+1)}\cdots \overset{\overset{0}{+}}{(n)}$. Since conjugation by elements of centralizer of $a$ preserves distance in \graph, without loss of generality we may take $g$ to be of the following form for some integer $r$ with $0 \leq r \leq l$ and  even integers $\lambda_i$:

$$g =  
\overset{\overset{0}{+}}{(1)}\cdots \overset{\overset{0}{+}}{(r)}\overset{\overset{\lambda_{r+1}}{-}}{(r+1)}\cdots \overset{\overset{\lambda_{r+k}}{-}}{(r+k)}
\overset{\overset{0}{+}}{(r+k+1)} \cdots \overset{\overset{0}{+}}{(n)}.$$
Now consider the following sequence, where $p=\lfloor\frac{k}{l}\rfloor$ and $\lambda_{r+k+1}$, $\ldots$, $\lambda_{n}$ are arbitrary even integers.
\begin{align*} g_0 &= \overset{\overset{0}{+}}{(1)}\cdots \overset{\overset{0}{+}}{(l)}\overset{\overset{\lambda_{l+1}}{-}}{(l+1)}\cdots \overset{\overset{\lambda_{n}}{-}}{(n)}
\\
g_1 &= \overset{\overset{0}{-}}{(1)}\cdots \overset{\overset{0}{-}}{(l)}\overset{\overset{0}{+}}{(l+1)}\cdots \overset{\overset{0}{+}}{(2l)}\overset{\overset{\lambda_{2l+1}}{-}}{(2l+1)}\cdots \overset{\overset{\lambda_{n}}{-}}{(n)}
\\
g_2 &= \overset{\overset{0}{-}}{(1)}\cdots \overset{\overset{0}{-}}{(2l)}\overset{\overset{0}{+}}{(2l+1)}\cdots \overset{\overset{0}{+}}{(3l)}\overset{\overset{\lambda_{3l+1}}{-}}{(3l+1)}\cdots \overset{\overset{\lambda_{n}}{-}}{(n)}
\\ 
&\vdots\\
g_{p-1}&=\overset{\overset{0}{-}}{(1)}\cdots \overset{\overset{0}{-}}{((p-1)l)}\overset{\overset{0}{+}}{((p-1)l+1)}\cdots \overset{\overset{0}{+}}{(pl)}\overset{\overset{\lambda_{pl+1}}{-}}{(pl+1)}\cdots \overset{\overset{\lambda_{n}}{-}}{(n)}\\
g_{p}&=\overset{\overset{0}{-}}{(1)}\cdots \overset{\overset{0}{-}}{(pl)}\overset{\overset{0}{+}}{(pl+1)}\cdots \overset{\overset{0}{+}}{((p+1)l)}\overset{\overset{\lambda_{(p+1)l+1}}{-}}{((p+1)l+1)}\cdots \overset{\overset{\lambda_{n}}{-}}{(n)}
\end{align*}
It is clear that $g_i$ commutes with $g_{i+1}$ for $0 \leq i < p$. Moreover $g$ commutes with $g_0$, and $g_p$ commutes with $a$. If $l$ divides $k$, then $g_p=a$, which implies that $d(g,a) \leq p+1 = \frac{k}{l} + 1 = \lceil\frac{n}{l}\rceil$. If $l$ does not divide $k$, then $d(g,a) \leq p + 2 = \lfloor\frac{k}{l}\rfloor + 2 = \lceil\frac{n}{l}\rceil$. The diameter of the graph in each case does equal this bound because at each stage of a path from $g$ to $a$ we can add at most $l$ to the number of correct negative 1-cycles, but this requires the element being considered to share no fixed points with $a$. Hence, for example the element  $$g=
\overset{\overset{2}{-}}{(1)}\cdots \overset{\overset{2}{-}}{(k)}\overset{\overset{0}{+}}{(k_e+k_o+1)}\cdots \overset{\overset{0}{+}}{(n)}$$
must be distance at least $\lceil \frac{n-l}{l} \rceil + 1 = \lceil \frac{n}{l}\rceil$ from $a$.\\

The remaining case is when $k_e$ and $k_o$ are both nonzero, and $2l < n$. \\

For any $x \in X$ define $c(x)$ to be the number of `correct' negative 1-cycles in $x$. That is, cycles $\overset{\overset{0}{-}}{(\alpha)}$ where $1 \leq \alpha \leq k_e$ or $\overset{\overset{1}{-}}{(\beta)}$ where $k_e < \beta \leq k_e + k_o$. Thus, for example, $c(a) = k_e + k_o$. We say that other negative 1-cycles are `incorrect'. 

Let $g \in X$. Then $g$ commutes with an element $x_0$ whose positive 1-cycles are $\overset{\overset{0}{+}}{(1)}\cdots \overset{\overset{0}{+}}{(l)}$. We will describe a sequence $x_0, x_1, \ldots$ where at each stage $x_i$ is an element of $X$ such that $c(x_i) \geq li$ and the positive 1-cycles of $x_i$ are $\overset{\overset{0}{+}}{(\alpha_1)}\cdots \overset{\overset{0}{+}}{(\alpha_r)}\overset{\overset{0}{+}}{(\beta_1)}\cdots \overset{\overset{0}{+}}{(\beta_s)}$ for some $\alpha_j, \beta_j$ where $1\leq \alpha_j \leq k_e$ and $k_e < \beta_j \leq k_e + k_o$, with $r+s = l$. Moreover, for $i > 0$, $x_i$ will commute with $x_{i-1}$.\\

Observe that the positive 1-cycles of $x_0$ have the required form, and $c(x_0) \geq 0\times l = 0$. Assume that we have $x_0, \ldots, x_i$ and let the positive 1-cycles of $x_i$ be $\overset{\overset{0}{+}}{(\alpha_1)}\cdots \overset{\overset{0}{+}}{(\alpha_r)}\overset{\overset{0}{+}}{(\beta_1)}\cdots \overset{\overset{0}{+}}{(\beta_s)}$. To form $x_{i+1}$ we look for incorrect cycles $\overset{\overset{\lambda_1}{-}}{(\gamma_1)}\cdots \overset{\overset{\lambda_r}{-}}{(\gamma_r)}\overset{\overset{\mu_1}{-}}{(\delta_1)}\cdots \overset{\overset{\mu_s}{-}}{(\delta_s)}$ of $x_i$, where each $\lambda_j$ is even, each $\mu_j$ is odd, $\gamma_j \leq n-l$ and $\delta_j \leq n-l$.  If such cycles can be found, then $x_i$ commutes with $x_{i+1}$ where $x_{i+1}$ is given by replacing the cycles $\overset{\overset{0}{+}}{(\alpha_1)}\cdots \overset{\overset{0}{+}}{(\alpha_r)}\overset{\overset{0}{+}}{(\beta_1)}\cdots \overset{\overset{0}{+}}{(\beta_s)}\overset{\overset{\lambda_1}{-}}{(\gamma_1)}\cdots \overset{\overset{\lambda_r}{-}}{(\gamma_r)}\overset{\overset{\mu_1}{-}}{(\delta_1)}\cdots \overset{\overset{\mu_s}{-}}{(\delta_s)}$ of $x_i$ with  $\overset{\overset{0}{-}}{(\alpha_1)}\cdots \overset{\overset{0}{-}}{(\alpha_r)}\overset{\overset{1}{-}}{(\beta_1)}\cdots \overset{\overset{1}{-}}{(\beta_s)}\overset{\overset{0}{+}}{(\gamma_1)}\cdots \overset{\overset{0}{+}}{(\gamma_r)}\overset{\overset{0}{+}}{(\delta_1)}\cdots \overset{\overset{0}{+}}{(\delta_s)}$, and leaving all other cycles unchanged. It is clear that $x_{i+1}$ commutes with $x_i$;  moreover $x_{i+1} \in X$, has the appropriate positive 1-cycles and $c(x_{i+1}) \geq (i+1)l$. 

This sequence can continue until we have some $x_i$ with cycles $\overset{\overset{0}{+}}{(\alpha_1)}\cdots \overset{\overset{0}{+}}{(\alpha_r)}\overset{\overset{0}{+}}{(\beta_1)}\cdots \overset{\overset{0}{+}}{(\beta_s)}$ but (without loss of generality) fewer than $r$ incorrect cycles $\overset{\overset{\lambda}{-}}{(\gamma)}$ where $\lambda$ is even and $\gamma \leq n-l$. Suppose there are exactly $t$ such cycles (with $t < r$). So $x_i$ has the cycles $\overset{\overset{\lambda_1}{-}}{(\gamma_1)}\cdots \overset{\overset{\lambda_t}{-}}{(\gamma_t)}$. Since $\alpha_j < k_e$ for each $j$ there must be $\ep_{t+1}, \ldots, \ep_n$ where $n-l < \ep_j \leq n$ for each $\ep_j$, and even numbers $\lambda_j$ such that $x_i$ has the cycles $\overset{\overset{\lambda_1}{-}}{(\gamma_1)}\cdots \overset{\overset{\lambda_t}{-}}{(\gamma_t)}\overset{\overset{\lambda_{t+1}}{-}}{(\ep_{t+1})}\cdots \overset{\overset{\lambda_r}{-}}{(\ep_r)}$. Similarly $x_i$ has cycles $\overset{\overset{\mu_1}{-}}{(\delta_1)}\cdots \overset{\overset{\mu_t'}{-}}{(\delta_t')}\overset{\overset{\mu_{t'+1}}{-}}{(\ep'_{t'+1})}\cdots \overset{\overset{\mu_r}{-}}{(\ep'_s)}$ for odd $\mu$, some $t' \leq s$, $\delta_j \leq n-l$ and $\ep'_j > n-l$.

We now define $y$ to be $x_{i}$ with $$\overset{\overset{0}{+}}{(\alpha_1)}\cdots \overset{\overset{0}{+}}{(\alpha_r)}\overset{\overset{0}{+}}{(\beta_1)}\cdots \overset{\overset{0}{+}}{(\beta_s)} \overset{\overset{\lambda_1}{-}}{(\gamma_1)}\cdots \overset{\overset{\lambda_r}{-}}{(\gamma_t)}\overset{\overset{\lambda_{t+1}}{-}}{(\ep_{t+1})}\cdots \overset{\overset{\lambda_r}{-}}{(\ep_r)}\overset{\overset{\mu_1}{-}}{(\delta_1)}\cdots \overset{\overset{\mu_t'}{-}}{(\delta_t')}\overset{\overset{\mu_{t'+1}}{-}}{(\ep'_{t'+1})}\cdots \overset{\overset{\mu_r}{-}}{(\ep'_s)}$$
replaced with
$$\overset{\overset{0}{-}}{(\alpha_1)}\cdots \overset{\overset{0}{-}}{(\alpha_r)}\overset{\overset{1}{-}}{(\beta_1)}\cdots \overset{\overset{1}{-}}{(\beta_s)} \overset{\overset{0}{+}}{(\gamma_1)}\cdots \overset{\overset{0}{+}}{(\gamma_t)}\overset{\overset{0}{+}}{(\ep_{t+1})}\cdots \overset{\overset{0}{+}}{(\ep_r)}\overset{\overset{0}{+}}{(\delta_1)}\cdots \overset{\overset{0}{+}}{(\delta_t')}\overset{\overset{0}{+}}{(\ep'_{t'+1})}\cdots \overset{\overset{0}{+}}{(\ep'_s)}.$$
Now $x_i$ commutes with $y$, and $c(y) \geq (i+1)l$. We observe that $d(g,y) \leq i+2$. Notice that every $\beta$ with $k_e < \beta \leq k_e + k_o$ is either a fixed point of $y$ or appears in a cycle $\overset{\overset{\mu}{-}}{(\beta)}$ with $\mu$ odd. Every incorrect even negative 1-cycle  of $y$ features a fixed point of $a$, so
for some $q \leq l$, and conjugating $y$ by a suitable element of the centralizer of $a$ if necessary, we can assume that the even negative 1-cycles of $y$ are
$$ \overset{\overset{\lambda_1}{-}}{(n-l+1)}\cdots \overset{\overset{\lambda_q}{-}}{(n-l+q)}\overset{\overset{0}{-}}{(q+1)}\cdots \overset{\overset{0}{-}}{(k_e)}.$$
Now $y$ has at least $(i+1)l$ correct negative 1-cycles. If we ignore the even negative 1-cycles and the correct odd negative 1-cycles, then $n - (i+1)l - q$ cycles remain (including $l$ fixed points). We can use the result for $k_e = 0$ on this remaining part of $y$ to see that $y$ is distance at most $\lceil\frac{n -  (i+1)l - q}{l}\rceil$
from the element $z$ of $X$ given by
$$\overset{\overset{\lambda_1}{-}}{(n-l+1)}\cdots \overset{\overset{\lambda_q}{-}}{(n-l+q)}\overset{\overset{0}{-}}{(q+1)}\cdots \overset{\overset{0}{-}}{(k_e)}\overset{\overset{0}{+}}{(1)}\cdots \overset{\overset{0}{+}}{(q)}\overset{\overset{0}{+}}{(n-l+q+1)}\cdots \overset{\overset{0}{+}}{(n)}\overset{\overset{1}{-}}{(k_e+1)}\cdots \overset{\overset{1}{-}}{(k_e+k_o)}.$$ 
Now $z$ commutes with $a$, and so 
$d(g,a) \leq d(g,y) + d(y,z) + 1 = i+2 +  \lceil\frac{n - q - (i+1)l}{l}\rceil + 1 = 2 + \lceil\frac{n-q}{l}\rceil$.
If $l = 1$ and $q = 0$, then in fact $z = a$ so $d(g,a) = n+1$. If $q=1$ then again $d(g,a) = n+1$. If $l > 1$ then we have $d(g,a) \leq \lceil \frac{n}{l}\rceil + 2$. To give a lower bound on the diameter consider
$$g=
\overset{\overset{2}{-}}{(1)}\cdots \overset{\overset{2}{-}}{(k_e)}\overset{\overset{3}{-}}{(k_e+1)}\cdots  \overset{\overset{3}{-}}{(k_e+k_o)}\overset{\overset{0}{+}}{(k_e+k_o+1)}\cdots \overset{\overset{0}{+}}{(n)}.$$ 
To create $l$ additional correct negative 1-cycles at each stage of a path from $g$ to $a$ one requires each fixed point to be a point not fixed by $a$; moreover in this case completion of the process for, say, the even negative 1-cycles requires the recreation of at least one fixed point between $n-l+1$ and $n$ and hence fewer than $l$ correct negative 1-cycles being created at the next stage. Thus when $l=1$ we have $d(g,a) \geq n+1$, and when $l > 1$ we have $d(g,a) \geq \lceil \frac{n}{l}\rceil$. This completes the proof. \end{proof}
   
 From now on, assume that $m>0$. Then we can take
   $$ a={\overset{\overset{0}{+}}{(1\; 2)}} \cdots \overset{\overset{0}{+}}{(2m-1\;2m)}
   \overset{\overset{0}{-}}{(2m+1)}\cdots \overset{\overset{0}{-}}{(2m+k_e)}
   \overset{\overset{1}{-}}{(2m+k_e+1)}\cdots 
   \overset{\overset{1}{-}}{(n-l)}\overset{\overset{0}{+}}{(n-l+1)}\cdots 
   \overset{\overset{0}{+}}{(n)},$$
   where $2m+k_e+k_o +l=n$.

\begin{prop}
\label{m>0k=0} Suppose $k_e = k_o = 0$. If $2m = n$, then $\diam \C(G,X) \leq 3$. If $n \geq 5$ and $l \geq 2$ then $\diam\graph \leq 5$.
\end{prop}

 \begin{proof} Let $x \in X$. We can write
 $a=(g,{\bf 0})$ and $x=(h,\bv)$, where $g, h$ are conjugate elements of the underlying Weyl group $W$. Define an element $y=(h',{\bf 0})$, by $h'=hw_0$, where $w_0$ is the unique central involution of $W$. The result is that every minus sign in $h$ corresponds to a plus sign in $h'$, and  every plus sign corresponds to a minus sign. Since $k=0$, $h$ and $h'$ are conjugate in $W$. So,  by Theorem \ref{finite}(i), $d(g,h')\leq 2$ when $2m=n$ and $d(g,h') \leq 4$ otherwise. Now $d(a,y) = d(g,h')$ and $y$ commutes with $x$ by Lemmas \ref{1cycle} and  \ref{2cycle}. The result follows immediately.  \end{proof}

\begin{prop} \label{mnonzerol=0} Suppose $m \geq 1$, $l=0$, one of $k_e$ and $k_o$ is zero, and either $m > 1$ or $\max\{k_e, k_0\} \geq 3$ (or both). Then \graph\ is connected with diameter at most $n$. If $x \in X$ and $d(x,a) = n$ then $m=1$ and the transposition of $x$ is $\overset{\overset{\lambda}{\ast}}{(1\; 2)}$ for some $\lambda$. \end{prop}
\begin{proof}
Let $x \in X$. Without loss of generality $k_o = 0$. We start with the case $m=1$. We will show that if the transposition of $x$ is $\overset{\overset{\lambda}{\ast}}{(1\; 2)}$ for some $\lambda$, then $d(x,a) \leq n$. Otherwise $d(x,a) \leq n-1$. The graph for $n=5$ (and hence $m=1$) is Figure \ref{l0}. We can see from this graph that the inductive hypothesis holds for $n=5$, so assume $n \geq 6$.  Suppose the transposition of $x$ contains some $\alpha$ with $\alpha > 2$. Then by Lemma \ref{2cycle}(iii) $x$ commutes with some $y \in X$ such that $y$ has the 1-cycle  $\overset{\overset{0}{-}}{(\alpha)}$ and such that the transposition of $y$ is not $\overset{\overset{\lambda}{\ast}}{(1\; 2)}$. Ignoring the cycle $\overset{\overset{0}{-}}{(\alpha)}$ we can work within $G_{\{1,\ldots, \alpha-1, \alpha+1, \ldots, n\}}$, to see that inductively $d(y,a) \leq n-2$. Hence $d(x,a) \leq n-1$. If the transposition of $x$ is $\overset{\overset{\lambda}{\ast}}{(1\; 2)}$ then $x$ certainly commutes with an element of $X$ which does not have this transposition. So $d(x,a) \leq n$ as required.\\ 

Now we assume $m \geq 2$ and proceed by induction on $k_e$ to show that $\diam\graph \leq n-1$. Suppose $k_e = 2$. Then $x$ is distance at most 2 from an element $y$ of $X$ which has the transposition $\overset{\overset{0}{+/-}}{(n-1 \;\; n)}$.  To see this, note that if both $n-1$ and $n$ appear in transpositions of $x$, or if both appear in 1-cycles of $x$, then Lemma \ref{2cycle} or Lemma \ref{doubletrans}, as appropriate, implies that $x$ commutes with some $x'$ in $X$ which contains a transposition of the form $\overset{\overset{\lambda}{-/+}}{(n-1\;\; n)}$ for some $\lambda$. Then $x'$ commutes with $y$. If on the other hand, $x$ contains (for example) the 1-cycle $\overset{\overset{\mu}{-}}{(n)}$ and $n-1$ appears in a transposition $\overset{\overset{\sigma}{+/-}}{(\ep\;\; n-1)}$ for some $\ep$ less than $n-1$, then $x$ commutes with $x''$ in $X$ containing the transpositions $\overset{\overset{\lambda}{-}}{(\ep\;\; n-1)}$ and $\overset{\overset{\sigma}{-}}{(\ep'\;\; n)}$ for some $\lambda$ and $\ep'$. Lemma \ref{doubletrans} now implies that $x$ commutes with an appropriate $y$, in particular one containing the transpositions $\overset{\overset{\sigma+\lambda}{-}}{(\ep\; \ep')}$ and $\overset{\overset{0}{-}}{(n-1\;\; n)}$. Now $y$ in turn commutes with some $z$ in $X$ with the 1-cycles $\overset{\overset{0}{-}}{(n-1)}$ and $\overset{\overset{0}{-}}{(n)}$. If we ignore these cycles and work in $G_{n-2}$, then Table \ref{tableG4} implies that when $n=6$ $d(z,a) \leq 2$, and when $n>6$ Proposition \ref{m>0k=0} tells us that $d(z, a) \leq 3$. Therefore $\diam\graph \leq n-1$. 

Finally, suppose $m\geq 2$ and $k_e > 2$.   Suppose there is some transposition of $x$ containing an element $\alpha$ with $\alpha > 2m$. Then by Lemma \ref{2cycle}(iii) $x$ commutes with some $y \in X$ such that $y$ has the 1-cycle  $\overset{\overset{0}{-}}{(\alpha)}$. By induction $d(y,a) \leq n-2$. Hence $d(x,a) \leq n-1$. The final possibility is that the elements of the transpositions of $x$ are $\{1, 2, \ldots, 2m\}$.  Since $m > 1$ we can use Lemma~\ref{doubletrans} to show that $x$ commutes with some $y$ in $X$ containing the transposition $\overset{\overset{0}{\ast}}{(1\; 2)}$. Working in $G_{\{3, 4, \ldots, n\}}$ (using the case $m=1$ and induction on $m$) we see that $d(y,a) \leq n-2$. Hence $d(x,a) \leq n-1$, which completes the proof of Proposition \ref{mnonzerol=0}.  
\end{proof}

\begin{lemma} \label{kois2}
	Suppose $m=1$, $l=0$, $k_o = 2$, $k_e \geq 3$ and $x \in X$. 
	Then $\graph$ is connected with diameter at most $n+1$. 
	\end{lemma}
	
	\begin{proof}
		Let $x \in X$. The distance of $x$ from $a$ will largely hinge on the whereabouts of $n$ and $n-1$. But first we deal with the case where the transposition of $x$ is $\overset{\overset{\lambda}{+/-}}{(12)}$. Here $x$ commutes with an element $y$ of $X$ having cycles $\overset{\overset{\mu}{-}}{(1)}\overset{\overset{\mu'}{-}}{(2)}$ with $\mu$ and $\mu'$ odd. Using Proposition \ref{mnonzerol=0} on the remaining cycles of $y$, we see that $y$ is distance at most $n-2$ from the element $b$ of $X$ whose cycles are the same as $a$ except that we have $
		\overset{\overset{\mu}{-}}{(1)}
		\overset{\overset{\mu'}{-}}{(2)}
		\overset{\overset{0}{+}}{(n-1 \; n)}$ instead of $
		\overset{\overset{1}{-}}{(n-1)}
		\overset{\overset{1}{-}}{(n)}
		\overset{\overset{0}{+}}{(1 \; 2)}$. Clearly $d(b,a) = 2$. Hence $d(x,a) \leq n+1$. This also shows that any element with odd negative 1-cycles $\overset{\overset{\mu}{-}}{(1)}\overset{\overset{\mu'}{-}}{(2)}$ is distance at most $n$ from $a$. Assume from now on that the transposition of $x$ is not  $\overset{\overset{\lambda}{+/-}}{(12)}$, and that its odd negative 1-cycles are not $\overset{\overset{\mu}{-}}{(1)}\overset{\overset{\mu'}{-}}{(2)}$.\\
		
		 If the transposition of $x$ is $\overset{\overset{\lambda}{+/-}}{(n-1 \; n)}$ then $x$ is distance 2 from some $y \in X$ with cycles $\overset{\overset{1}{-}}{(n-1)}\overset{\overset{1}{-}}{(n)}$. Ignoring these cycles we use Proposition \ref{mnonzerol=0} in $G_{n-2}$ to see that $d(y,a) \leq n-2$ if the transposition of $y$ is $\overset{\overset{\lambda}{+/-}}{(12)}$ and $d(y,a) \leq n-3$ otherwise. Thus $d(x,a) \leq n$. 
		
	
	If $x$ has cycles $\overset{\overset{\mu}{-}}{(n-1)}\overset{\overset{\mu'}{-}}{(n)}$ where $\mu \equiv \mu' \mod{2}$, then $x$ is distance 3 from some $y$ in $X$ with cycles $\overset{\overset{1}{-}}{(n-1)}\overset{\overset{1}{-}}{(n)}$. Thus $d(x,a) \leq n+1$ if the transposition of $x$ is $\overset{\overset{\lambda}{+/-}}{(12)}$ and $d(x,a) \leq n$ otherwise. 
	
	If one of $n-1$ and $n$, say $n$, appears in the transposition of $x$ and the other appears in an even negative 1-cycle $\overset{\overset{\lambda}{-}}{(n-1)}$, then $x$ commutes with some $x' \in X$ with the cycles  $\overset{\overset{\lambda}{-}}{(n-1)}\overset{\overset{\lambda}{-}}{(n)}$. Then $x'$ is distance 2 from some $y \in X$ with the cycles $\overset{\overset{1}{-}}{(n-1)}\overset{\overset{1}{-}}{(n)}$, and such that the transposition of $y$ is not $\overset{\overset{\sigma}{+/-}}{(12)}$ for any $\sigma$. By Proposition \ref{mnonzerol=0} again, $d(x,a) \leq n$.
	
	If $x$ contains  $\overset{\overset{\lambda}{-}}{(n-1)}\overset{\overset{\mu}{-}}{(n)}\overset{\overset{2\sigma - \mu}{-}}{(\alpha)}$ where $\lambda$ is even, $\mu$ is odd, $\sigma$ is an integer and $\alpha < n-1$, then $x$ commutes with an element of $X$ containing $\overset{\overset{\lambda}{-}}{(n-1)}\overset{\overset{\sigma}{-}}{(n\; \alpha)}$, which commutes with an element of $X$ containing $\overset{\overset{\lambda}{-}}{(n-1)}\overset{\overset{\lambda}{-}}{(n)}\overset{\overset{2\sigma - \lambda}{-}}{(\alpha)}$, which commutes with an element of $X$ containing $\overset{\overset{0}{-}}{(n-1\; n)}\overset{\overset{2\sigma - \lambda}{-}}{(\alpha)}$, which finally commutes with an element $y$ of $X$ containing $\overset{\overset{1}{-}}{(n-1)}\overset{\overset{1}{-}}{(n)}\overset{\overset{2\sigma - \lambda}{-}}{(\alpha)}$, such that the transposition of $y$ is the same as the transposition of $x$, namely not $\overset{\overset{\lambda}{+/-}}{(1\; 2)}$. By Proposition~\ref{mnonzerol=0}, $d(y,a) \leq n-3$. Hence $d(x,a) \leq n+1$.
	
	The final case to consider is where $n$ is contained in the transposition of $x$ and $n-1$ is in an odd negative 1-cycle. So $x$ contains $\overset{\overset{\lambda}{-}}{(\alpha \; n)}\overset{\overset{\mu}{-}}{(\beta)}\overset{\overset{2\sigma - \mu}{-}}{(n-1)}$ for some $\alpha, \beta$ and integers $\lambda, \mu, \sigma$ with $\mu$ odd. If $\beta \notin \{1,2\}$ then subject to appropriate conjugation we can set $\beta = 3$. Using Proposition~\ref{mnonzerol=0} in $G_{\{1,\ldots, n\}\setminus\{3, n-1\}}$ we see that $x$ is distance at most $n-3$ from the element $y$ where $y = \overset{\overset{0}{+}}{(1  2)}\overset{\overset{2\sigma}{-}}{(n)}\overset{\overset{0}{-}}{(4)}\cdots \overset{\overset{0}{-}}{(n-2)}\overset{\overset{\mu}{-}}{(3)}\overset{\overset{2\sigma - \mu}{-}}{(n-1)}$. Then $y$ commutes with $\overset{\overset{\sigma}{-}}{(3\;  n-1)}\overset{\overset{2\sigma}{-}}{(n)}\overset{\overset{0}{-}}{(4)}\cdots \overset{\overset{0}{-}}{(n-2)}\overset{\overset{1}{-}}{(1)}\overset{\overset{1}{-}}{(2)}$, which commutes with $\overset{\overset{0}{+}}{(4  5)}\overset{\overset{2\sigma}{-}}{(n)}\overset{\overset{2\sigma}{-}}{(n-1)}\overset{\overset{0}{-}}{(3)}\overset{\overset{0}{-}}{(6)}\cdots \overset{\overset{0}{-}}{(n-2)}\overset{\overset{1}{-}}{(1)}\overset{\overset{1}{-}}{(2)}$ which is distance 2 from $a$. Thus $d(x,a) \leq n+1$. The last case is where without loss of generality $\beta = 1$ and we can assume $\alpha$ is 2 or 5. Let $\alpha'$ be the other element of $\{2,5\}$. Then $x$ is distance 2 from some $x'$ in $X$ containing $\overset{\overset{\sigma}{-}}{(1\; n-1)}\overset{\overset{\tau}{-}}{(\alpha)}\overset{\overset{\tau'}{-}}{(\alpha')}\overset{\overset{\nu}{-}}{(3)}\overset{\overset{2\kappa-\nu}{-}}{(4)}$ where $\tau'$ is determined by $x$ but we may choose $\tau$ arbitrarily, and $\nu, \kappa$ are integers with $\nu$ odd. Now let $y = \overset{\overset{0}{+}}{(n-1\;  n)}\overset{\overset{0}{-}}{(1)}\overset{\overset{0}{-}}{(2)}\overset{\overset{0}{-}}{(5)}\cdots \overset{\overset{0}{-}}{(n-2)}\overset{\overset{\nu}{-}}{(3)}\overset{\overset{2\kappa - \nu}{-}}{(4)}$. Then $d(x,x') = 2$ and $d(y,a) = 3$. What is $d(x',y)$? If $k_e > 3$ then $n\geq 8$. Set $\tau = 0$. Now working in $G_{\{1,\ldots, n\}\setminus \{\alpha, 3, 4\}}$ we see from Proposition~\ref{mnonzerol=0} that $d(x',y) \leq n-4$. If $k_e = 3$ then $n=7$. This time set $\tau = \tau'$. Then $x'$ commutes with $\overset{\overset{0}{+}}{(2 5)}\overset{\overset{0}{-}}{(1)}\overset{\overset{2\sigma}{-}}{(6)}\overset{\overset{2\sigma'}{-}}{(7)}\overset{\overset{\nu}{-}}{(3)}\overset{\overset{2\kappa-\nu}{-}}{(4)}$ for some $\sigma'$, which commutes with $\overset{\overset{\sigma+\sigma'}{-}}{(6 7)}\overset{\overset{0}{-}}{(1)}\overset{\overset{0}{-}}{(2)}\overset{\overset{0}{-}}{(5)}\overset{\overset{\nu}{-}}{(3)}\overset{\overset{2\kappa-\nu}{-}}{(4)}$ which commutes with $y$, so $d(x',y) = 3$. Hence in all cases $d(x,a) \leq n+1$, which completes the proof of Lemma \ref{kois2}.
\end{proof}

   \begin{thm}\label{connectm=1l=0} Suppose $n \geq 7$, $m \geq 1$, $l=0$ and $k_e$ and $k_o$ are both at least 2. Then \graph\ is connected with diameter at most $n+2$. If $m\geq 2$, then $\diam\graph\leq n$.    
 \end{thm}           
        
       \begin{proof} 
Assume that $k_e \geq k_o$. Suppose first that $m=1$ and let $x \in X$. We use induction on $k_o$ to show that if the transposition of $x$ is not ${\overset{\overset{\lambda}{+/-}}{(1\; 2)}}$, then $d(x,a) \leq n+1$. Otherwise $d(x,a) \leq n+2$.  If $k_o = 2$ the result holds by Lemma~\ref{kois2}. If $k_o > 2$ and the transposition of $x$ is not ${\overset{\overset{\lambda}{+/-}}{(1\; 2)}}$, then $x$ contains cycles ${\overset{\overset{\nu}{+/-}}{(\alpha \beta)}}{\overset{\overset{\lambda}{-}}{(\gamma)}}{\overset{\overset{\lambda'}{-}}{(\delta)}}$ where $\alpha > 2$, $\gamma > 2$,  ${\overset{\overset{\sigma}{-}}{(\alpha)}}$ is a cycle of $a$ and $\lambda, \lambda'$ and $\sigma$ are all congruent modulo 2. Now $x$ commutes with some $x' \in X$ containing the cycles ${\overset{\overset{\sigma}{-}}{(\alpha)}}{\overset{\overset{\nu'}{-}}{(\gamma\; \delta)}}$ for appropriate $\nu'$. If we ignore ${\overset{\overset{\sigma}{-}}{(\alpha)}}$ and work in $G_{\{1, \ldots, n\}\setminus\{\alpha\}}$, then inductively $d(x',a) \leq n$. Hence $d(x,a) \leq n+1$. Suppose the transposition of $x$ is ${\overset{\overset{\lambda}{+/-}}{(1\; 2)}}$. Then
$x$ commutes with some $y$ in $X$ that does not contain this transposition, and we have seen that $d(y,a) \leq n+1$. Hence $d(x,a) \leq n+2$. This completes the case $m=1$.\\

If $m > 1$ and $k_o = 2$ then it is easy to see that $x$ is distance at most 2 from an element of $X$ containing the transposition $\overset{\overset{0}{+}}{(n-1\;\; n)}$. Thus $x$ is distance at most 3 from an element $y$ of $X$ containing $\overset{\overset{1}{-}}{(n-1)}
\overset{\overset{1}{-}}{(n)}$. Ignoring these 1-cycles we may work in $G_{n-2}$ and apply Proposition~\ref{mnonzerol=0} to see that $d(z,a) \leq n-3$. Hence $d(x,a) \leq n$. \\


Now suppose $m>1$ and $k_o > 2$. If $x$ has a transposition containing an element $\alpha$ with $\alpha > 2m$, then $x$ commutes with an element $y$ containing $\overset{\overset{0}{-}}{(\alpha)}$ or $\overset{\overset{1}{-}}{(\alpha)}$ (choose whichever of these is a cycle of $a$). Then we can ignore this cycle and work in $G_{\{1,\ldots, n\}\setminus \{\alpha\}}$. Inductively, using the base case $k_o = 2$, we see that $d(y,a) \leq n-1$. Hence $d(x,a) \leq n$. Finally we deal with the case that every transposition of $x$ is of the form ${\overset{\overset{\lambda}{+/-}}{(\alpha\; \beta)}}$ where $\alpha < \beta \leq 2m$. Because $k_e \geq k_o \geq 3$, it must be the case that $x$ contains cycles: ${\overset{\overset{\sigma_1}{+/-}}{(\alpha_1\; \alpha_2)}}{\overset{\overset{\sigma_2}{+/-}}{(\alpha_3\; \alpha_4)}}\overset{\overset{\lambda_1}{-}}{(\beta_1)}\overset{\overset{\lambda_2}{-}}{(\beta_2)}\overset{\overset{\mu_1}{-}}{(\gamma_1)}\overset{\overset{\mu_2}{-}}{(\gamma_2)}$ where $\lambda_1 \equiv \lambda_2 \mod 2$, $\mu_1 \equiv \mu_2 \mod 2$, $\{\alpha_1, \ldots, \alpha_4\} \subseteq \{1, \ldots, 2m\}$, $2m < \beta_1 < \beta_2 \leq 2m+k_e$ and $2m+k_e < \gamma_1 < \gamma_2 \leq n$. Then $x$ is distance 2 from an element $y$ with the cycles $\overset{\overset{0}{-}}{(\beta)}$ and $\overset{\overset{1}{-}}{(\gamma)}$. Now, working inductively in $G_{\{1, \ldots, n\} \setminus \{\beta, \gamma\}}$ we see that $d(y,a) \leq n-2$. Hence $d(x,a) \leq n$.  \end{proof}
 
 \begin{lemma}
 	\label{n=5} If $n=1$, $m=1$, $l=1$ and $k_e = 2$, then $\diam \graph = 5$.
 \end{lemma}
 
 \begin{proof}
 	Let $x \in X$. If the transposition is  $\overset{\overset{\lambda}{+/-}}{(1 2)}$ then using Table \ref{tableG3} for $G_{\{3,4,5\}}$ we see that $d(x,a) \leq 3$. Suppose the transposition of $x$ is  ${\overset{\overset{\lambda}{+/-}}{(\alpha \; \beta)}}$ where $\{\alpha, \beta\} \subseteq \{3,4,5\}$, and let $\gamma$ be the remaining element of $\{3,4,5\}$. Then $x$ commutes with $x' =  {\overset{\overset{0}{+/-}}{(\alpha \; \beta)}} {\overset{\overset{\lambda}{-}}{(1)}} {\overset{\overset{\lambda'}{-}}{(2)}} {\overset{\overset{0}{+}}{(\gamma)}}$ for some even integers $\lambda$ and $\lambda'$. Now $x'$ commutes with   ${\overset{\overset{(\lambda+\lambda')/2}{-}}{(1 2)}} {\overset{\overset{0}{-}}{(\alpha)}} {\overset{\overset{0}{-}}{(\beta)}} {\overset{\overset{0}{+}}{(\gamma)}}$, which commutes with $a$. So $d(x,a) \leq 3$. The remaining cases are (interchanging 1 and 2 if necessary) when $x =  {\overset{\overset{\sigma}{+/-}}{(1\; \alpha)}} {\overset{\overset{\lambda}{-}}{(2)}} {\overset{\overset{\lambda'}{-}}{(\beta)}} {\overset{\overset{0}{+}}{(\gamma)}}$   or ${\overset{\overset{\sigma}{+/-}}{(1\; \alpha)}} {\overset{\overset{+}{+}}{(2)}} {\overset{\overset{\lambda'}{-}}{(\beta)}} {\overset{\overset{\lambda}{-}}{(\gamma)}}$ for appropriate $\sigma, \lambda$ and $\lambda'$. The following is a path of length at most 5 from either of these to $a$:$\quad$ $x$ , ${\overset{\overset{0}{-/+}}{(1\; \alpha)}} {\overset{\overset{0}{+}}{(2)}} {\overset{\overset{\lambda'}{-}}{(\beta)}} {\overset{\overset{\lambda}{-}}{(\gamma)}}$,  
 		${\overset{\overset{(\lambda'+\lambda)/2}{-}}{(\beta\; \gamma)}} {\overset{\overset{0}{-}}{(1)}} {\overset{\overset{0}{+}}{(2)}} {\overset{\overset{0}{-}}{(\alpha)}}$,  
    		${\overset{\overset{0}{+}}{(\beta\; \gamma)}} {\overset{\overset{0}{-}}{(1)}} {\overset{\overset{0}{-}}{(2)}} {\overset{\overset{0}{+}}{(\alpha)}}$, 	${\overset{\overset{0}{+}}{(1\; 2)}} {\overset{\overset{0}{-}}{(\beta)}} {\overset{\overset{0}{-}}{(\gamma)}} {\overset{\overset{0}{+}}{(\alpha)}}$, $a$. Hence in all cases $d(x,a) \leq 5$, which completes the proof.
 \end{proof}
 We observe, because we will need it for Lemma \ref{n=6} later, that the proof of Lemma \ref{n=5} shows that $d(x,a) \leq 4$ in all cases except where (modulo interchanging 1 and 2, or 3 and 4) the transposition of $x$ is ${\overset{\overset{0}{+}}{(1\; 3)}}$.
 
 \begin{thm}
 	\label{m>0l>0} Suppose $m \geq 1$, $l \geq 1$ and $\max\{k_e, k_o, l\} \geq 2$. Then $\graph$ is connected with diameter at most $n$. 
 \end{thm}
 
 \begin{proof} Suppose $n$ is minimal such that $\graph$ is a counterexample, and let $x \in X$ such that $d(x,a) > n$. By Lemma \ref{n=5} we can assume $n \geq 6$. If $l \geq 2$ then $x \in X$ commutes with some $y \in X$ containing $\overset{\overset{0}{+}}{(n)}$. Ignoring this 1-cycle we can work in $G_{n-1}$ to find a path to $a$, which inductively is of length at most $n-1$, which implies $d(x,a) \leq n$, contrary to our choice of $x$. Hence $l = 1$. \\
 
If elements $\alpha$ and $\beta$ lying between $2m + 1$ and $2m+k_e$ are contained in transpositions of $x$, then $x$ commutes with some $x' \in X$ having the transposition ${\overset{\overset{\lambda}{+/-}}{(\alpha \; \beta)}}$ for some $\lambda$ (if $m=1$ then we can set $x=x'$). If we ignore this transposition of $x'$ we can work in $G_{\{1, \ldots, n\}\setminus \{\alpha, \beta\}}$, which is either the case $l=1$ with a smaller $m$, so inductively the graph has diameter at most $n-2$, or (if $m=1$) we can use Theorem \ref{disc}, in which case the graph has diameter $n-1$.
In either case, we see that $x$ is distance at most $n-1$ from the element $b$ of $X$ whose cycles are the same as $a$ except that $b$ has $\overset{\overset{0}{-}}{(1)}\overset{\overset{0}{-}}{(2)}\overset{\overset{0}{+}}{(\alpha \; \beta)}$ instead of $\overset{\overset{0}{-}}{(12)}\overset{\overset{0}{-}}{(\alpha)}\overset{\overset{0}{-}}{(\beta)}$. Since $b$ commutes with $a$ we have $d(x,a) \leq n$. The same argument holds if $2m+k_e < \alpha < \beta < n$. If $1 \leq \alpha < \beta \leq 2m$ then similar reasoning shows again $x$ is distance at most $n-1$ from an element $b$ containing $\overset{\overset{0}{+}}{(\alpha \; \beta)}$ that commutes with $a$. So $d(x,a) \leq n$. Thus none of these pairs $\alpha, \beta$ exist in $x$. This implies $m \leq 2$. Moreover if $m=2$ then $k_0 \neq 0$ and without loss of generality the transpositions of $x$ contain $1, 5, n-1$ and $n$. \\

If there is some $\beta$ in a transposition of $x$ with $2m+k_e < \beta < n$ and if $k_o \geq 2$, then $x$ commutes with some $y \in X$ containing $\overset{\overset{1}{-}}{(\beta)}$. Inductively we can work in $G_{\{1,\ldots, n\} \setminus\{\beta\}}$ to see that $d(y,a) \leq n-1$. Hence $d(x,a) \leq n$. Similarly, as long as either $k_e > 2$ or $k_o \geq 2$ (or both), if there is some $\alpha$ in a transposition of $x$ with $2m < \alpha \leq 2m+k_e$ then inductively $d(x,a) \leq n$. This means that if $m=2$ then $k_e = 2$ and $k_o = 1$. \\

Suppose that $m=2$, $k_e = 2$ and $k_o = 1$, so that $n=8$. We have observed that the transpositions of $x$ must contain $1, 5, 7$ and $8$. If 6 is not contained in an even negative 1-cycle of $x$, then $x$ commutes with some $y$ containing a transposition  ${\overset{\overset{\lambda}{+/-}}{(\alpha \; \beta)}}$ for some $\lambda$, where $\{\alpha \beta\} \subset \{2,3,4\}$. As at the start of this proof, $d(y,a) \leq 7$. More explicitly, inductively $y$ is distance at most 6 from an element $b$ containing ${\overset{\overset{0}{+}}{(\alpha \; \beta)}}$ that commutes with $a$. Hence $d(x,a) \leq 8$, a contradiction. Thus $m=1$. \\

If $x$ contains $\overset{\overset{0}{+}}{(\alpha)}$ where $2+k_e < \alpha < n$ then $x$ commutes with some $x' \in X$ containing $\overset{\overset{1}{-}}{(\alpha)}$. So, using the result for $G_{\{1,\ldots, n\} \setminus \{\alpha\}}$, we get $d(x',a) \leq n-1$. Thus $d(x,a) \leq n$. The same reasoning holds if $k_e > 2$ and $x$ contains $\overset{\overset{0}{+}}{(\alpha)}$ for some $\alpha$ where $2m < \alpha \leq 2m+k_e$. Suppose first that $k_o \neq 1$. Since $n \geq 6$ we must have $k_o \geq 2$ or $k_o = 0$ and $k_e \geq 3$. Thus the transposition of $x$ cannot contain any $\alpha$ with $2m < \alpha < n$. Hence without loss of generality $x$ contains $\overset{\overset{\lambda}{*}}{(1\;n)}\overset{\overset{0}{+}}{(2)}$. 
 	 Suppose that $x$ contains $\overset{\overset{\sigma}{-}}{(\beta)}
 	  \overset{\overset{\sigma'}{-}}{(\beta')}$   where
 	   $\sigma \equiv \sigma'$ mod ${2}$ and either $3\leq \beta < \beta' \leq k_e + 2$ or $k_e+3\leq \beta < \beta' < n$. Then, $x$ commutes with some $y \in X$ where $y$ contains $\overset{\overset{\lambda}{-}}{(\beta\;\beta')} \overset{\overset{\sigma}{-}}{(1)}$
 	   for some $\lambda$. Now $a$ commutes with some element $z$ containing $\overset{\overset{\lambda}{-}}{(\beta\;\beta')} \overset{\overset{\sigma}{-}}{(1)}\overset{\overset{\sigma}{-}}{(2)}$. Using Theorem~\ref{disc} on $G_{\{1, \ldots, n\}\setminus\{1, \beta, \beta'\}}$ (that is, removing $\overset{\overset{\lambda}{-}}{(\beta\;\beta')} \overset{\overset{\alpha}{-}}{(1)}$, from $y$ and $z$) we see that $d(y,z) \leq n-2$.
 	  Therefore, $d(x,a) \leq n$. Finally, we are reduced to the possibility that $x$ contains $\overset{\overset{\lambda}{*}}{(1\;n)} \overset{\overset{0}{+}}{(2)}$ and any pair $\overset{\overset{\sigma}{-}}{(\beta)} \overset{\overset{\sigma'}{-}}{(\beta')}$ where $\beta < \beta'$ and $\sigma \equiv \sigma' \mod{2}$ satisfies $3\leq \beta \leq k_e + 2 < \beta' < n$. Since $n>5$ and $k_o \neq 1$, the only way this can occur is when $k_e = k_o = 2$ and $n=7$. Conjugating by an element of the centraliser of $a$ if necessary, we can assume that $x= \overset{\overset{\lambda}{*}}{(1\;7)} \overset{\overset{\sigma}{-}}{(3)}\overset{\overset{\sigma'}{-}}{(5)} \overset{\overset{\mu}{-}}{(4)}\overset{\overset{\mu'}{-}}{(6)} \overset{\overset{0}{+}}{(2)}$ where $\sigma$ and $\sigma'$ are even, and $\mu$, $\mu'$ are odd. The following is a path from $x$ to $a$ in the graph: $x$, $\overset{\overset{\sigma - \sigma'}{+}}{(3\;5)} \overset{\overset{0}{-}}{(1)}\overset{\overset{\lambda'}{-}}{(7)} \overset{\overset{\mu}{-}}{(4)}\overset{\overset{\mu'}{-}}{(6)} \overset{\overset{0}{+}}{(2)}$, 
 	  $\overset{\overset{0}{-}}{(3\;5)} \overset{\overset{0}{-}}{(1)}\overset{\overset{0}{-}}{(2)} \overset{\overset{\mu}{-}}{(4)}\overset{\overset{\mu'}{-}}{(6)} \overset{\overset{0}{+}}{(7)}$, 
 	$\overset{\overset{0}{+}}{(1\;2)} \overset{\overset{0}{-}}{(3)}\overset{\overset{0}{-}}{(5)} \overset{\overset{\mu'}{-}}{(6)}\overset{\overset{1}{-}}{(7)} \overset{\overset{0}{+}}{(4)}$, 
 	$\overset{\overset{0}{+}}{(1\;2)} \overset{\overset{0}{-}}{(3)}\overset{\overset{0}{-}}{(4)} \overset{\overset{\mu'}{-}}{(6)}\overset{\overset{1}{-}}{(7)} \overset{\overset{0}{+}}{(5)}$, 
 	$\overset{\overset{0}{+}}{(1\;2)} \overset{\overset{0}{-}}{(3)}\overset{\overset{0}{-}}{(4)} \overset{\overset{\mu'}{-}}{(5)}\overset{\overset{1}{-}}{(7)} \overset{\overset{0}{+}}{(6)}$, $a$. So $d(x,a) \leq 6$, another contradiction.\\

 The final case to consider is where $m=1$ and $k_o = 1$. Assume first that $n\geq 7$. Suppose that $x$ contains $\overset{\overset{\lambda}{-}}{(1)}
  	  \overset{\overset{\lambda'}{-}}{(2)}$ with $\lambda, \lambda'$ both even. Then $x$ commutes with some $y \in X$ containing $\overset{\overset{\nu}{-}}{(1\;2)}$ for some $\nu$. As observed at the start of this proof, $d(y,a) \leq n-1$. Hence $d(x,a) \leq n$. Now suppose $x$ does not contain $\overset{\overset{\lambda}{-}}{(1)}
  	  \overset
  	  {\overset{\lambda'}{-}}{(2)}$. Then there is a set  $A = \{\alpha_1, \ldots, \alpha_6\}$ containing 1, 2 and $n-1$ such that $x=\overset{\overset{\sigma}{+/-}}
  	  {(\alpha_1\;\alpha_2)}\overset{\overset{\lambda}{-}}{(\alpha_3)}\overset{\overset{\lambda'}{-}}{(\alpha_4)}\overset{\overset{\mu}{-}}{(\alpha_5)}\overset{\overset{0}{+}}{(\alpha_6)}\bar x$  where $\lambda$ and $\lambda'$ are both even and $\mu$ is odd. By replacing $x$ with a conjugate under the centraliser of $a$, we can further assume that $A = \{1,2,3,4,n-1,\beta\}$ for some $\beta$. Let $z = 
  	  \overset{\overset{0}{+}}{(1\;2)}
  	  \overset{\overset{0}{-}}{(3)}
  	  \overset{\overset{0}{-}}{(4)}
  	  \overset{\overset{1}{-}}{(n-1)}
  	  \overset{\overset{0}{+}}{(\beta)}\bar x$ . Now, using the case $k_o=1, n=6$ on $G_{\{1, 2, 3, 4, n-1, \beta\}}$ we see that  $d(x,z) \leq 6$. Next we apply Theorem \ref{disc} to $\overset{\overset{0}{+}}{(\beta)}x'$ and $\overset{\overset
  	  { 0}{-}}{(5)}\overset{\overset{0}{-}}
  	  {(6)}\dots \overset{\overset{0}{-}}{(n-2)}
  	    \overset {\overset{0}{+}}{(n)} \in G_{\{5, \dots, n-2, n\}}$ (noting that here there are no odd negative 1-cycles) to see that $d(z,a) \leq n-5$. Hence $d(x,a) \leq 6 + n-5 = n-1$, another contradiction.\\
  	    
  	   The remaining possibility is that $n=6$, $m=1$, $k_e = 2, k_o = 1$. But Lemma \ref{n=6} immediately after this proof shows that this graph has diameter 6, which is the final contradiction completing the proof of Theorem \ref{m>0l>0}.
 	\end{proof}
 	
 \begin{lemma}\label{n=6}
 	If $n=6$, $m=1$, $k_e = 2$ and $k_o = 1$ then $\diam \graph \leq 6$.
 \end{lemma}
 
     \begin{proof}
     	We have $a = \overset{\overset{0}{+}}{(1\;2)}
     	  	  \overset{\overset{0}{-}}{(3)}
     	  	  \overset{\overset{0}{-}}{(4)}
     	  	  \overset{\overset{1}{-}}{(5)}
     	  	  \overset{\overset{0}{+}}{(6)}$. Let $x \in X$. Suppose $x$ contains $\overset{\overset{\sigma}{+/-}}{(1\;2)}$ for some $\sigma$. By Theorem \ref{disc}, the graph of $\overset{\overset{0}{-}}{(3)}\overset{\overset{0}{-}}{(4)}\overset{\overset{1}{-}}{(5)}\overset{\overset{0}{+}}{(6)}$ in $G_{\{3,4,5,6\}}$ has diameter 5. Hence $d(x,a) \leq 5$. If $x$ contains $\overset{\overset{\lambda}{-}}{(1)}\overset{\overset{\lambda'}{-}}{(2)}$ where $\lambda, \lambda'$ are even, then $x$ commutes with some $x'$ containing $\overset{\overset{\sigma}{+/-}}{(1\;2)}$ for some $\sigma$, and we have just seen that $d(x',a) \leq 5$. Hence $d(x,a) \leq 6$. \\
We next consider the cases where 1 and 2 are in different types of 1-cycle of $x$. Because interchanging 1 and 2 does not affect the distance of $x$ from $a$, there are just three cases to consider here: $x$ contains $\overset{\overset{\lambda}{-}}{(1)}\overset{\overset{0}{+}}{(2)}$, $\overset{\overset{\lambda}{-}}{(1)}\overset{\overset{\mu}{-}}{(2)}$ or $\overset{\overset{\mu}{-}}{(1)}\overset{\overset{0}{+}}{(2)}$, where $\lambda$ is even and $\mu$ is odd. In the first case we have $x = \overset{\overset{\sigma}{+/-}}{(\alpha \; \beta)} \overset{\overset{\lambda}{-}}{(1)}
\overset{\overset{\lambda'}{-}}{(\gamma)}
 \overset{\overset{\mu}{-}}{(\delta)}
 \overset{\overset{0}{+}}{(2)}$ where $\{\alpha, \beta, \gamma, \delta\} = \{3,4,5,6\}$, $\lambda, \lambda'$ are even and $\mu$ is odd. Here $x$ commutes with $\overset{\overset{0}{-/+}}{(\alpha \; \beta)} \overset{\overset{\lambda}{-}}{(1)}  \overset{\overset{\lambda}{-}}{(2)}
  \overset{\overset{\mu}{-}}{(\delta)} \overset{\overset{0}{+}}{(\gamma)}$, which commutes with the element $b$ given by $b = \overset{\overset{0}{+}}{(1\; 2)} \overset{\overset{0}{-}}{(\alpha)} \overset{\overset{0}{-}}{(\beta)}\overset{\overset{1}{-}}{(\gamma)}
  \overset{\overset{0}{+}}{(\delta)}$. Glancing at Table \ref{tableG4} we observe that in $G_{\{3,4,5,6\}}$ elements at distance 4 or 5 from $\overset{\overset{0}{-}}{(3)}\overset{\overset{0}{-}}{(4)}\overset{\overset{1}{-}}{(5)}\overset{\overset{0}{+}}{(6)}$ require at least one even label to be nonzero. So $d(b,a) \leq 3$. hence $d(x,a) \leq 5$. For the second case, where $x$ contains $\overset{\overset{\lambda}{-}}{(1)}\overset{\overset{\mu}{-}}{(2)}$, note that $x$ commutes with some $x'$ containing $\overset{\overset{\lambda}{-}}{(1)}\overset{\overset{0}{+}}{(2)}$. Hence, using the first case we get $d(x,a) \leq 6$. The third case is where $x$ contains $\overset{\overset{\mu}{-}}{(1)}\overset{\overset{0}{+}}{(2)}$. Here, $x$ commutes with some $x'$ of the form $\overset{\overset{\sigma}{+/-}}{(\alpha \; \beta)} \overset{\overset{\lambda}{-}}{(\gamma)}
  \overset{\overset{\lambda'}{-}}{(5)}
   \overset{\overset{\mu}{-}}{(1)}
   \overset{\overset{0}{+}}{(2)}$. The following is a path from $x'$ to $a$:
   $$x' = \overset{\overset{\sigma}{+/-}}{(\alpha \; \beta)} \overset{\overset{\lambda}{-}}{(\gamma)}
     \overset{\overset{\lambda'}{-}}{(5)}
      \overset{\overset{\mu}{-}}{(1)}
      \overset{\overset{0}{+}}{(2)}, \overset{\overset{0}{-/+}}{(\alpha \; \beta)} \overset{\overset{\lambda}{-}}{(\gamma)}
     \overset{\overset{0}{-}}{(2)}
      \overset{\overset{\mu}{-}}{(1)}
      \overset{\overset{0}{+}}{(5)}, \overset{\overset{0}{+}}{(\alpha \; \beta)} \overset{\overset{\lambda}{-}}{(\gamma)}
           \overset{\overset{0}{-}}{(2)}
            \overset{\overset{1}{-}}{(5)}
            \overset{\overset{0}{+}}{(1)}, \overset{\overset{0}{+}}{(\alpha \; \beta)} \overset{\overset{0}{-}}{(1)}
\overset{\overset{0}{-}}{(2)} \overset{\overset{1}{-}}{(5)}
 \overset{\overset{0}{+}}{(\gamma)}, \overset{\overset{0}{+}}{(1 \; 2)} \overset{\overset{0}{-}}{(\alpha)}\overset{\overset{0}{-}}{(\beta)}
\overset{\overset{1}{-}}{(5)}  \overset{\overset{0}{+}}{(\gamma)}, a.$$ So $d(x,a) \leq 6$. \\

It remains to take care of the possibility that exactly one of 1 and 2 lies in the transposition of $x$. Here, without loss of generality, we can assume $x$ contains $ \overset{\overset{\sigma}{+/-}}{(1 \; \alpha)}$ for $\alpha > 2$. If $x$ contains $\overset{\overset{\mu}{-}}{(2)}$ for $\mu$ odd, or $\overset{\overset{0}{+}}{(2)}$, then $x$ commutes with some $x'$ containing $\overset{\overset{\lambda}{-}}{(1)}\overset{\overset{0}{+}}{(2)}$, for even $\lambda$, and as shown earlier in this proof, $d(x',a) \leq 5$. Hence $d(x,a) \leq 6$. We may thus assume that $x = \overset{\overset{\sigma}{+/-}}{(1 \; \alpha)}\overset{\overset{\lambda}{-}}{(2)}\overset{\overset{\lambda'}{-}}{(\beta)}\overset{\overset{\mu}{-}}{(\gamma)}\overset{\overset{0}{+}}{(\delta)}$. If $\delta = 5$ then $x$ commutes with $x'$ given by $\overset{\overset{\sigma}{+/-}}{(1 \; \alpha)}\overset{\overset{\lambda}{-}}{(2)}\overset{\overset{\lambda'}{-}}{(\beta)}\overset{\overset{1}{-}}{(5)}\overset{\overset{0}{+}}{(\gamma)}$. By Theorem \ref{n=5} applied to $G_{\{1,2,3,4,6\}}$ we see that $d(x',a) \leq 5$ and so $d(x,a) \leq 6$. So we can assume that $\delta \neq 5$. Now $x$ is distance 4 from the element $b$ given by $b = \overset{\overset{0}{+}}{(1 \; 2)}\overset{\overset{0}{-}}{(\beta)}\overset{\overset{0}{-}}{(\delta)}\overset{\overset{1}{-}}{(\alpha)}\overset{\overset{0}{+}}{(\gamma)}$, as shown by the following path.
$$x = \overset{\overset{\sigma}{+/-}}{(1 \; \alpha)}\overset{\overset{\lambda}{-}}{(2)}\overset{\overset{\lambda'}{-}}{(\beta)}\overset{\overset{\mu}{-}}{(\gamma)}\overset{\overset{0}{+}}{(\delta)}, \overset{\overset{0}{-/+}}{(1 \; \alpha)}\overset{\overset{\lambda'}{-}}{(\delta)}\overset{\overset{\lambda'}{-}}{(\beta)}\overset{\overset{\mu}{-}}{(\gamma)}\overset{\overset{0}{+}}{(2)}, 
\overset{\overset{0}{+}}{(\beta \; \delta)}\overset{\overset{0}{-}}{(1)}\overset{\overset{0}{-}}{(\alpha)}\overset{\overset{\mu}{-}}{(\gamma)}\overset{\overset{0}{+}}{(2)}, 
\overset{\overset{0}{+}}{(\beta \; \delta)}\overset{\overset{0}{-}}{(1)}\overset{\overset{0}{-}}{(2)}\overset{\overset{\mu}{-}}{(\gamma)}\overset{\overset{0}{+}}{(\alpha)}, \overset{\overset{0}{+}}{(1 \; 2)}\overset{\overset{0}{-}}{(\beta)}\overset{\overset{0}{-}}{(\delta)}\overset{\overset{1}{-}}{(\alpha)}\overset{\overset{0}{+}}{(\gamma)} = b.$$
If $\alpha = 5$, then $d(b,a) = 1$ so $d(x,a) \leq 5$. If $\alpha = 6$, then $d(b,a) \leq 2$, so $d(x,a) \leq 6$. If $\gamma = 5$, then $d(b,a) \leq 2$ so $d(x,a) \leq 6$. Since interchanging 3 and 4 does not affect the distance from $a$, there is only one case left to deal with: $\alpha = 3$ and $\beta = 5$. So $x = \overset{\overset{\sigma}{+/-}}{(1 \;3)}\overset{\overset{\lambda}{-}}{(2)}\overset{\overset{\lambda'}{-}}{(5)}\overset{\overset{\mu}{-}}{(\gamma)}\overset{\overset{0}{+}}{(\delta)}$. But $x$ is the same distance from $a$ as $y = \overset{\overset{\sigma}{+/-}}{(2 \; 3)}\overset{\overset{\lambda}{-}}{(1)}\overset{\overset{\lambda'}{-}}{(5)}\overset{\overset{\mu}{-}}{(\gamma)}\overset{\overset{0}{+}}{(\delta)}$, and $y$ commutes with $z$ given by $z = \overset{\overset{\sigma'}{+/-}}{(1 \; 5)}\overset{\overset{\lambda}{-}}{(2)}\overset{\overset{\lambda''}{-}}{(3)}\overset{\overset{\mu}{-}}{(\gamma)}\overset{\overset{0}{+}}{(\delta)}$ for appropriate $\lambda''$ and $\sigma'$. By the `$\alpha = 5$' case above, $d(z,a) \leq 5$. Therefore $d(x,a) = d(y,a) \leq d(z,a) + 1 \leq 6$. This completes the proof. 
     \end{proof}
     
\paragraph{Proof of Theorems \ref{summary} and \ref{main}} 
Suppose $m=0$. If $l=0$ then \graph\ is disconnected by Theorem \ref{mainhalf}(i). Otherwise $\diam\graph\leq n+1$ by Theorem \ref{disc}. Suppose $m=1$ and $l=0$. If  one of $k_e$ and $k_o$ is 1 or the largest of $k_e$ and $k_o$ is 2, then \graph\ is disconnected by Theorem \ref{mainhalf} (ii), (iv) and (v). Otherwise $\diam\graph \leq n+2$ by Propositions \ref{m>0k=0}, \ref{mnonzerol=0}, Lemma \ref{kois2} and Theorem \ref{connectm=1l=0}. Suppose $m>1$ and $l=0$. If either of $k_e$ or $k_o$ is 1, then \graph\ is disconnected by Theorem \ref{mainhalf}(ii). Otherwise $\diam \graph \leq n$ by Propositions \ref{m>0k=0}, \ref{mnonzerol=0} and Theorem \ref{connectm=1l=0}. Finally suppose $m \geq 1$ and $l>0$. If $\max\{k_e,k_o,l\} = 1$ then \graph\ is disconnected by Theorem \ref{mainhalf}(iii). Otherwise $\diam\graph\leq n$ by Lemma \ref{n=5}, Theorem \ref{m>0l>0} and Lemma \ref{n=6}.  \qed

\section{Examples}\label{examples}

We first summarise the information on commuting involution graphs for $G_2$, $G_3$ and $G_4$. For fixed $a \in X$, the $i^{\mathrm{th}}$ disc $\Delta_i(a)$ is the set of elements of $X$ which are distance $i$ from $a$.  Since a length preserving graph automorphism interchanges classes with $k_e \geq k_o$ and those with $k_e \leq k_o$, we list here only those classes with $k_e \geq k_o$. In describing the elements we omit positive 1-cycles. In $G_2$ there are two connected graphs, both of diameter 2. If $a = \overset{\overset{0}{-}}{(1)}$ then $\Delta_1(a) = \{\overset{\overset{\lambda}{-}}{(2)}: \lambda \text{ even}\}$ with the remaining elements of $X$ comprising the second disc. If $a =  \overset{\overset{0}{+}}{(12)}$ then $\Delta_1(a) = \{\overset{\overset{\sigma}{-}}{(12)} : \sigma \in \zz\}$, with all other elements of $X$ lying in the second disc. Tables \ref{tableG3} and \ref{tableG4} give, for each connected graph and each disc, a list of orbit representatives under the action of $C_G(a)$. 
 We use $\lambda, \lambda'$ and so on to represent arbitrary even numbers, with $[\lambda]$ being an arbitrary nonzero even number. We use $\mu, \mu'$ and so on for arbitrary odd numbers, with $[\mu]$ being any odd number other than 1. Finally $\sigma$, $\sigma'$ and so on will be arbitrary integers with $[\sigma]$ an arbitrary nonzero integer. 
\begin{table}[h!]
  \begin{center}
  \begin{tabular}{|c|c|c|c|c|} \hline
 $a$ & $\Delta_1$ &  $\Delta_2$ & $\Delta_3$ & $\Delta_4$ \\
  \hline
   $\overset{\overset{0}{-}}{(1)}$ &  $\overset{\overset{\lambda}{-}}{(2)}$ & $\overset{\overset{[\lambda]}{-}}{(1)}$ & & \\
  \hline
   $\overset{\overset{0}{-}}{(1)}\overset{\overset{0}{-}}{(2)}$  & $\overset{\overset{0}{-}}{(1)}\overset{\overset{\lambda}{-}}{(3)}$ & $\overset{\overset{\lambda'}{-}}{(2)}\overset{\overset{\lambda}{-}}{(3)}$, $\overset{\overset{0}{-}}{(1)}\overset{\overset{[\lambda']}{-}}{(2)}$ & $\overset{\overset{[\lambda]}{-}}{(1)}\overset{\overset{[\lambda']}{-}}{(2)}$ & \\
   \hline
   $\overset{\overset{0}{-}}{(1)}\overset{\overset{1}{-}}{(2)}$ &  $\overset{\overset{0}{-}}{(1)}\overset{\overset{\mu}{-}}{(3)}$, $\overset{\overset{1}{-}}{(2)}\overset{\overset{\lambda}{-}}{(3)}$ &  $\overset{\overset{0}{-}}{(1)}\overset{\overset{\mu}{-}}{(2)}$,
   $\overset{\overset{\lambda}{-}}{(2)}\overset{\overset{\mu}{-}}{(3)}$,
   $\overset{\overset{\lambda}{-}}{(1)}\overset{\overset{1}{-}}{(2)}$,
   $\overset{\overset{\mu}{-}}{(1)}\overset{\overset{\lambda}{-}}{(3)}$ 
    &  $\overset{\overset{\mu}{-}}{(1)}\overset{\overset{\lambda}{-}}{(2)}$,  $\overset{\overset{[\lambda]}{-}}{(1)}\overset{\overset{\mu}{-}}{(3)}$,  $\overset{\overset{[\mu]}{-}}{(2)}\overset{\overset{\lambda}{-}}{(3)}$ & 
     $\overset{\overset{[\lambda]}{-}}{(1)}\overset{\overset{[\mu]}{-}}{(2)}$\\ \hline
  \end{tabular}\caption{Connected Graphs for $G_3$ (with $k_e \geq k_o$)}\label{tableG3}
 \end{center} \end{table}

 \newcommand{\gappys}{\hspace*{-3mm}}
 \newcommand{\gappy}{\hspace*{-4mm}}
 
\begin{table}[h!]
  \begin{center}
   \begin{tabular}{|c|c|c|c|c|c|}
 \hline $a$ & $\Delta_1$ &  $\Delta_2$ & $\Delta_3$ & $\Delta_4$ & $\Delta_5$ \\
   \hline
   $\overset{\overset{0}{-}}{(1)}$ &  $\overset{\overset{\lambda}{-}}{(2)}$ & $\overset{\overset{[\lambda]}{-}}{(1)}$ & & & \\     \hline
$\overset{\overset{0}{-}}{(1)}\overset{\overset{0}{-}}{(2)}$ & $\overset{\overset{0}{-}}{(1)}\overset{\overset{\lambda}{-}}{(3)}$,  $\overset{\overset{\lambda}{-}}{(3)}\overset{\overset{\lambda'}{-}}{(4)}$ & $X \setminus (\Delta_1(a) \cup \{a\})$ & & & \\   
         \hline
$\overset{\overset{0}{-}}{(1)}\overset{\overset{0}{-}}{(2)}\overset{\overset{0}{-}}{(3)}$  & $\overset{\overset{0}{-}}{(1)}\overset{\overset{0}{-}}{(2)}\overset{\overset{\lambda}{-}}{(4)}$ & \gappys $\overset{\overset{0}{-}}{(1)}\overset{\overset{\lambda}{-}}{(3)}\overset{\overset{\lambda'}{-}}{(4)}$, $\overset{\overset{0}{-}}{(1)}\overset{\overset{0}{-}}{(2)}\overset{\overset{[\lambda]}{-}}{(3)}$ \gappy & \gappys $\overset{\overset{\lambda}{-}}{(2)}\overset{\overset{\lambda'}{-}}{(3)}\overset{\overset{\lambda''}{-}}{(4)}$, $\overset{\overset{0}{-}}{(1)}\overset{\overset{[\lambda]}{-}}{(2)}\overset{\overset{[\lambda']}{-}}{(3)}$ \gappy & \gappys $\overset{\overset{[\lambda]}{-}}{(1)}\overset{\overset{[\lambda']}{-}}{(2)}\overset{\overset{[\lambda'']}{-}}{(3)}$ \gappy & \\   
             \hline
             $\overset{\overset{0}{-}}{(1)}\overset{\overset{1}{-}}{(2)}$&
\hspace*{-2mm} $\overset{\overset{0}{-}}{(1)}\overset{\overset{\mu}{-}}{(3)}$   $\overset{\overset{1}{-}}{(2)}\overset{\overset{\lambda}{-}}{(3)}$, $\overset{\overset{\lambda}{-}}{(3)}\overset{\overset{\mu}{-}}{(4)}$ \hspace*{-2mm} & $X \setminus (\Delta_1(a) \cup \{a\})$& & &  \\   
                 \hline
$\overset{\overset{0}{-}}{(1)}\overset{\overset{0}{-}}{(2)}\overset{\overset{1}{-}}{(3)}$ &
\hspace*{-2mm} $\overset{\overset{0}{-}}{(1)}\overset{\overset{0}{-}}{(2)}\overset{\overset{\mu}{-}}{(4)}$,
$\overset{\overset{0}{-}}{(1)}\overset{\overset{1}{-}}{(3)}\overset{\overset{\lambda}{-}}{(4)}$ \hspace*{-2mm}
 & \begin{tabular}{c} \gappys
 $\overset{\overset{0}{-}}{(1)}\overset{\overset{\lambda}{-}}{(3)}\overset{\overset{\mu}{-}}{(4)}$,
 $\overset{\overset{0}{-}}{(1)}\overset{\overset{0}{-}}{(2)}\overset{\overset{[\mu]}{-}}{(3)}$\gappy \\ \gappys
 $\overset{\overset{0}{-}}{(1)}\overset{\overset{\mu}{-}}{(2)}\overset{\overset{\lambda}{-}}{(4)}$,
 $\overset{\overset{[\lambda]}{-}}{(1)}\overset{\overset{1}{-}}{(3)}\overset{\overset{\lambda'}{-}}{(4)}$ \gappy\\ \gappys
 $\overset{\overset{[\lambda]}{-}}{(1)}\overset{\overset{0}{-}}{(2)}\overset{\overset{1}{-}}{(3)}$ \gappy \end{tabular}& 
\begin{tabular}{c}
 \gappys $\overset{\overset{0}{-}}{(1)}\overset{\overset{\mu}{-}}{(2)}\overset{\overset{\lambda}{-}}{(3)}$, 
$\overset{\overset{[\lambda]}{-}}{(1)}\overset{\overset{[\lambda']}{-}}{(2)}\overset{\overset{1}{-}}{(3)}$ \gappy\\ \gappys
$\overset{\overset{[\lambda]}{-}}{(1)}\overset{\overset{\mu}{-}}{(2)}\overset{\overset{\lambda'}{-}}{(4)}$,
$\overset{\overset{[\lambda]}{-}}{(1)}\overset{\overset{0}{-}}{(2)}\overset{\overset{\mu}{-}}{(4)}$ \gappy\\ \gappys
$\overset{\overset{0}{-}}{(1)}\overset{\overset{[\mu]}{-}}{(3)}\overset{\overset{\lambda}{-}}{(4)}$,
$\overset{\overset{\mu}{-}}{(1)}\overset{\overset{\lambda}{-}}{(3)}\overset{\overset{\lambda'}{-}}{(4)}$ \gappy \\ \gappys
$\overset{\overset{[\lambda]}{-}}{(1)}\overset{\overset{\lambda'}{-}}{(3)}\overset{\overset{\mu}{-}}{(4)}$	\gappy \end{tabular}
& 
\begin{tabular}{c} \gappys
$\overset{\overset{0}{-}}{(1)}\overset{\overset{[\lambda]}{-}}{(2)}\overset{\overset{[\mu]}{-}}{(3)}$ \gappy \\ \gappys
$\overset{\overset{[\lambda]}{-}}{(1)}\overset{\overset{\mu}{-}}{(2)}\overset{\overset{\lambda'}{-}}{(3)}$\gappy \\ \gappys
$\overset{\overset{[\lambda]}{-}}{(1)}\overset{\overset{[\lambda']}{-}}{(2)}
\overset{\overset{\mu}{-}}{(4)}$\gappy \\ \gappys
$\overset{\overset{[\lambda]}{-}}{(1)}\overset{\overset{[\mu]}{-}}{(3)}\overset{\overset{\lambda'}{-}}{(4)}$ \gappy
\end{tabular}& \hspace*{-1mm}$\overset{\overset{[\lambda]}{-}}{(1)}\overset{\overset{[\lambda']}{-}}{(2)}
\overset{\overset{[\mu]}{-}}{(3)}$\hspace*{-1mm}\\   
                     \hline
 $\overset{\overset{0}{+}}{(12)}\overset{\overset{0}{+}}{(34)}$  & \hspace*{-5mm} \begin{tabular}{cc} \gappys $\overset{\overset{\sigma}{-}}{(12)}\overset{\overset{0}{+}}{(34)}$, $\overset{\overset{\sigma}{-}}{(12)}\overset{\overset{\sigma'}{-}}{(34)}$\gappy\\ \gappys $\overset{\overset{\sigma}{+}}{(13)}\overset{\overset{\sigma}{+}}{(24)}$, $\overset{\overset{\sigma}{-}}{(13)}\overset{\overset{\sigma}{-}}{(24)}$\gappy \end{tabular} \hspace*{-5mm} & $X \setminus (\Delta_1(a) \cup \{a\})$  &  & & \\   
       \hline
  \end{tabular}\caption{Connected Graphs for $G_4$ (with $k_e \geq k_o$)}\label{tableG4}
   \end{center} \end{table}

 Figure \ref{l0} is the collapsed adjacency graph for $a = \overset{\overset{0}{+}}{(12)}
 \overset{\overset{0}{-}}{(3)}\overset{\overset{0}
 	{-}}{(4)}\overset{\overset{0}{-}}{(5)}$ in $G_5$, which has diameter 5.
 In the graph $\sigma$ is an arbitrary integer, $[\sigma]$ is an arbitrary non-zero integer and $\lambda, \lambda'$ and  $\lambda''$ are arbitrary even integers. For simplicity we have only included shortest paths -- that is, we have omitted edges between nodes in the same disc.

  \begin{figure}[h!]
  \includegraphics[width=0.9\textwidth]{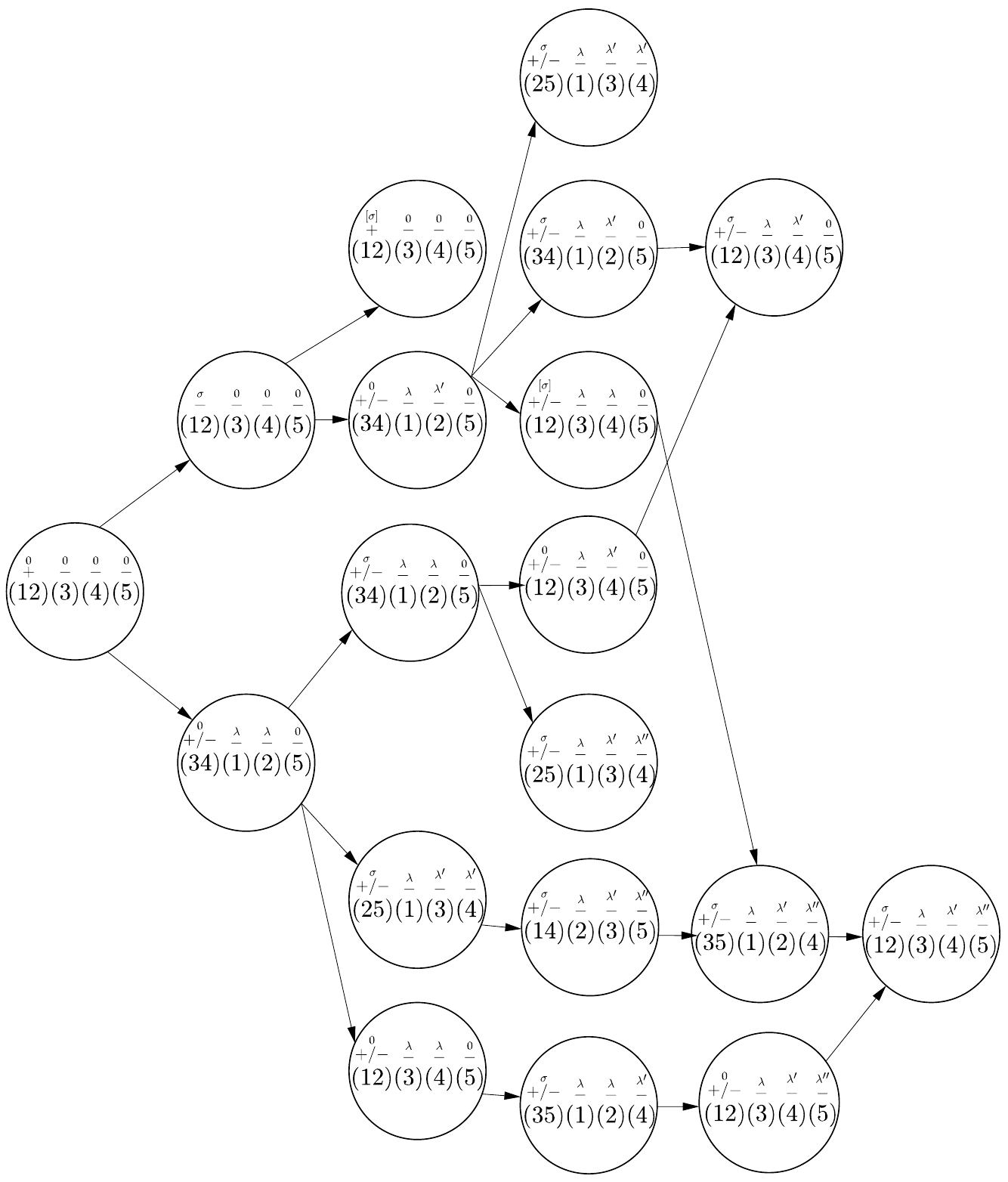}
   \caption{$n=5$, $k=3$ and $m=1$} \label{l0}
   \end{figure}

  \afterpage{\clearpage}

  \end{document}